\documentclass{aims}
\usepackage{amsmath}
  \usepackage{paralist}
  \usepackage{graphics} 
  \usepackage{epsfig} 
 \usepackage[colorlinks=true]{hyperref}
\hypersetup{urlcolor=blue, citecolor=red}

\def\currentvolume{X}
 \def\currentissue{X}
  \def\currentyear{200X}
   \def\currentmonth{XX}
    \def\ppages{X--XX}

\numberwithin{equation}{section}

\newtheorem{theorem}{Theorem}[section]
\newtheorem{corollary}{Corollary}

\newtheorem{lemma}[theorem]{Lemma}

\theoremstyle{definition}

\title[Frequency locking of modulated waves]
      {Frequency locking of modulated waves}

\author[L. Recke, A. Samoilenko et al.]{}

\subjclass{Primary: 34C30, 34C14, 34C15; Secondary: 34C29, 34C60. 34D35, 34D06}
 \keywords{Frequency locking, modulated waves, synchronization.}

 \email{recke@math.hu-berlin.de}
 \email{sam@imath.kiev.ua}
 \email{teplinsky@imath.kiev.ua}
 \email{vitk@imath.kiev.ua}
 \email{yanchuk@math.hu-berlin.de}

\thanks{ L.R. and S.Y. acknowledges the support of DFG Research Center \textsc{Matheon}
``Mathematics for key technologies'' under the projects D8 and D21.
A.S., A.T. and V.T. acknowledge the support of DFG cooperation project between
Germany and Ukraine 436UKR113/100/0-1.}

\begin{document}
\maketitle

\centerline{\scshape Lutz Recke }
\medskip
{\footnotesize
 \centerline{Institute of Mathematics, Humboldt University of Berlin,}
   \centerline{Unter den Linden 6, 10099 Berlin, Germany}
} 

\medskip

\centerline{\scshape Anatoly Samoilenko, Alexey Teplinsky and Viktor Tkachenko,}
\medskip
{\footnotesize
 \centerline{ Institute of Mathematics, National Academy of Sciences of Ukraine}
   \centerline{3 Tereschen\-kivska St., 01601 Kiev, Ukraine}
}

\medskip
\centerline{\scshape Serhiy Yanchuk}
\medskip
{\footnotesize
 \centerline{Institute of Mathematics, Humboldt University of Berlin,}
   \centerline{Unter den Linden 6, 10099 Berlin, Germany}
} 

\bigskip

 \centerline{(Communicated by the associate editor name)}

\begin{abstract}
We consider the behavior of a modulated wave solution to
an $\mathbb{S}^1$-equivariant autonomous system of differential equations under an external
forcing of modulated wave type. The modulation frequency of the forcing is assumed to be close to the
modulation frequency of the modulated wave solution, while the  wave frequency of the forcing is supposed to be far
from that of the modulated wave solution. We describe the domain in the three-dimensional
control parameter space (of frequencies and amplitude of the forcing)
where stable locking of the modulation frequencies of the forcing and the modulated wave solution
occurs.

Our system is a simplest case scenario for the behavior of  self-pulsating lasers under the influence of external
periodically modulated
optical signals.
\end{abstract}

\section{Introduction}

This paper investigates systems of differential equations of the type
\begin{eqnarray}
\frac{dx}{dt} & = & f(x)+g(x)|y|^{2},\label{01}\\
\frac{dy}{dt} & = & h(x)y+\gamma e^{i\alpha t}a(\beta t),\label{02}
\end{eqnarray}
where $x\in\mathbb{R}^{n},$ $y\in\mathbb{C}$, the functions  $f,g:\mathbb{R}^{n}\to\mathbb{R}^{n},$
$h:\mathbb{R}^{n}\to\mathbb{C}$, and $a:\mathbb{R}\to\mathbb{C}$
are sufficiently smooth of class $C^{l}$ with some positive
integer $l$. The function $a$ is $2\pi$-periodic, and $\alpha>0$, $\beta>0$
and $\gamma\ge0$ are parameters.
We assume that for $\gamma=0$ the unperturbed system
\begin{eqnarray}
\frac{dx}{dt} & = & f(x)+g(x)|y|^{2},\label{eq:010}\\
\frac{dy}{dt} & = & h(x)y,\label{eq:020}
\end{eqnarray}
has an exponentially orbitally stable quasi-periodic solution
of modulated wave type
\begin{eqnarray}
x(t)=x_{0}(\beta_{0}t),\ y(t)=y_{0}(\beta_{0}t)e^{i\alpha_{0}t}.\label{qp}
\end{eqnarray}
Here $\alpha_{0}>0$ and $\beta_{0}>0$  are constants, while
$x_{0}:\mathbb{R} \to \mathbb{R}^n$ and $y_{0}:\mathbb{R} \to \mathbb{C}$ are smooth $2\pi$-periodic functions.
We assume that the following nondegeneracy condition holds:
\begin{equation}
\label{>0}
\mbox{rank}
\left[\begin{array}{cc}
x_0'(\psi)&  0\\
 \Re y_0'(\psi)& -\Im  y_0(\psi)\\
 \Im y_0'(\psi)& \Re y_0(\psi)
\end{array}\right]=2.
\end{equation}
It is easy to verify that (\ref{>0}) is true for all $\psi \in \mathbb{R}$ if it is true for one $\psi$.
Moreover, without loss of generality we assume
that $\psi \mapsto \arg y_{0}(\psi)$ is periodic, i.e.\ the curve $y=y_{0}(\psi)$
in $\mathbb{C}$ does not loop around the origin (otherwise we should replace $y_0(\beta_0 t)$
by $y_0(\beta_0 t) e^{ik\beta_0 t}$ and $\alpha_0$ by $\alpha_0-k\beta_0$ with an appropriate $k \in \mathbb{Z}$).

It follows from assumption (\ref{>0}) that the set
\[
\mathcal{T}_{2}:=
\{(x_{0}(\psi),y_{0}(\psi)e^{i\varphi})\in\mathbb{R}^{n}\times\mathbb{C}:\;\varphi,\psi\in\mathbb{T}_{1}\},\]
where $\mathbb{T}_{1}=\mathbb{R}/(2\pi\mathbb{Z})$ is the unit circle,
is diffeomorphic to a two-dimensional torus.
Obviously, $\mathcal{T}_{2}$ is invariant with respect to the flow of  (\ref{eq:010})--(\ref{eq:020}),
and the solution (\ref{qp}) lies on  $\mathcal{T}_{2}$.

Roughly speaking, our main result describes the domain in the three-dimensional space of the
control parameters $\alpha$, $\beta$ and $\gamma$ with $|\alpha-\alpha_0|\gg 1$ and $\beta \approx \beta_0$
such that the following holds: For almost any solution $(x(t),y(t))$ to  (\ref{01})--(\ref{02}),
which is at a certain moment
close to  $\mathcal{T}_{2}$, there exists $\sigma \in \mathbb{R}$ such that
$$
\|x(t)-x_0(\beta t+\sigma)\|+\bigl||y(t)|-|y_0(\beta t+\sigma)|\bigr| \approx 0 \mbox{ for large } t.
$$

Let us reformulate our result in a more abstract language as well as in the language of a
physical application.

Abstractly speaking,  (\ref{eq:010})--(\ref{eq:020}) is an autonomous system which is equivariant
under the  $\mathbb{T}_{1}$-action $(x,y) \mapsto(x,e^{i\varphi}y)$, $\varphi \in  \mathbb{T}_{1}$,
on the phase space. The solution (\ref{qp}) is a so-called modulated wave solution or relative periodic
orbit to the  $\mathbb{T}_{1}$-equivariant system  (\ref{eq:010})--(\ref{eq:020}). It is well-known that generically
those solutions are structurally stable under small perturbations that do not destroy the autonomy and the
$\mathbb{T}_{1}$-equivariance of the system. Thus, our results describe the behavior of exponentially orbitally stable
modulated wave solutions to  $\mathbb{T}_{1}$-equivariant systems under external forcings of modulated wave
type in the case when the difference between the internal and the external modulation frequencies $\beta-\beta_0$
is small while
the difference between the internal and the external wave frequencies $\alpha-\alpha_0$ is large.
Note that in \cite{Recke1998}
related results are described for the case when both differences of modulation and wave frequencies are small,
and \cite{Recke1998a} considers the case when the internal state as well as the external
forcing are not modulated.
For an even more abstract setting of these results see \cite{Chillingworth2000}.

System  (\ref{01})--(\ref{02}) is a paradigmatic model for the dynamical behavior
of self-pulsating lasers under the influence of external periodically modulated optical signals.
For more involved mathematical models see, e.g.,
\cite{Bandelow1998, Lichtner2007,Nizette2001, Peterhof1999, Radziunas2006,Sieber2002,Wieczorek2005}
and for related experimental results see
\cite{Feiste1994,Sartorius1998}. In  (\ref{01})--(\ref{02}), the state
variables $x$ and $y$ describe the electron density and the optical field of the laser, respectively.
In particular, the absolute value $|y|$ describes the intensity of  the optical field.
The  $\mathbb{T}_{1}$-equivariance of  (\ref{eq:010})--(\ref{eq:020}) is the result of the
invariance of autonomous optical models with respect to shifts of optical phases.
The solution (\ref{qp}) describes a so-called self-pulsating state of the laser
in the case when the laser is driven by electric currents which are constant in time.
In those states the electron density and the
intensity of  the optical field are time periodic with the same frequency.
Self-pulsating states usually appear as a result of Hopf bifurcations from so-called continuous wave states,
where the  electron density and the
intensity of  the optical field are constant in time.

The structure of our paper is as follows. The main results are formulated in
Sec.~\ref{sec:Main-result}. The proof is splitted into four sections.
In Sec.~\ref{sec:Averaging} we use averaging transformations \cite{Bogoliubov1961}
in order to eliminate the fast oscillating terms with the frequency
$\alpha$. It appears that the first non-vanishing terms after the
averaging procedure are of order $\gamma^{2}/\alpha^{2}$. Local coordinates
in the vicinity of the stable invariant toroidal manifold are introduced
in Sec.~\ref{sec:Local-coordinates} and then
in Sec.~\ref{sec:Existence} the existence of perturbed manifold is proved.
The global behavior of a
system on the perturbed torus is described in Sec.~\ref{sec:Investigation-of-the}.
Among others, the methods of perturbation theory \cite{Samoilenko1991,Samoilenko2005}
are used in our analysis.

\section{Main results\label{sec:Main-result}}

In new coordinates $x=x,y=re^{i\varphi}$, $r,\varphi\in\mathbb{R}$,
the unperturbed system (\ref{eq:010})--(\ref{eq:020}) has the
form \begin{eqnarray}
 &  & \frac{dx}{dt}=f(x)+g(x)r^{2},\label{eq:UNPx}\\
 &  & \frac{dr}{dt}=\Re h(x)r,\label{eq:UNPy}\\
 &  & \frac{d\varphi}{dt}=\Im h(x). \label{eq:UNPphi}
\end{eqnarray}
This system has, by assumption, the two-frequency solution
\begin{eqnarray*}
x(t)=x_{0}(\beta_{0}t),\ r(t)=r_{0}(\beta_{0}t):=|y_{0}(\beta_{0}t)|,\
\varphi(t)=\alpha_{0}t+\arg y_{0}(\beta_{0}t).
\end{eqnarray*}
The subsystem (\ref{eq:UNPx})--(\ref{eq:UNPy}) does not depend on $\varphi$
and has an exponentially orbitally stable periodic solution
$
x(t)=x_{0}(\beta_{0}t),\ r(t)=r_{0}(\beta_{0}t).
$
The corresponding variational system has the following form
\begin{eqnarray}
\frac{dz}{d\psi}=A(\psi)z,\quad & z\in\mathbb{R}^{n+1},\label{eq:VAR}\end{eqnarray}
where
\[
A(\psi):=\frac{1}{\beta_0}\left[\begin{array}{cc}
{\displaystyle f'(x_{0}(\psi))+g'(x_{0}(\psi))r_{0}^{2}(\psi)} & {\displaystyle 2g(x_{0}(\psi))r_{0}(\psi)}\\[2mm]
{\displaystyle \Re h'(x_{0}(\psi))r_{0}(\psi)} &
{\displaystyle \Re h(x_{0}(\psi))}\end{array}\right].\]
We assume that
\begin{equation}
\label{ass}
\left\{
\begin{array}{l}
\mbox{the trivial multiplier $1$ of the  monodromy matrix of linear periodic}\\
\mbox{system (\ref{eq:VAR}) has multiplicity one, and the absolute values of all}\\
\mbox{other multipliers are less than $1$.}
\end{array}
\right.
\end{equation}
The adjoint system
$$\frac{dp}{d\psi}=-A^{T}(\psi)p,$$
has a nontrivial periodic solution
$p(\psi)$ ($A^T$ denotes the transpose of $A$), which can be normalized such that 
$$p^{T}(\psi) \left[\begin{array}{c}
x'_{0}(\psi)\\
r'_{0}(\psi)\end{array}\right] =1 \mbox{ for all } \psi.
$$

Let us define the function
$\mathcal{G}:\mathbb{R}^n\times \mathbb{T}_1 \to \mathbb{R}^{n+1}$ as follows
\begin{equation}
\label{mathcalG}
{\mathcal G}(x,\psi):=
\left[\begin{array}{c}
g(x)|a(\psi)|^{2}\\
0\end{array}\right].
\end{equation}

Our first result describes the behavior (under the perturbation by the forcing term with  $\gamma>0$) 
of $\mathcal{T}_{2}\times \mathbb{R}$, which is an integral manifold to (\ref{01})--(\ref{02}) with 
$\gamma=0$, as well as the dynamics of the system (\ref{01})--(\ref{02}) on the perturbed manifold.

\begin{theorem}
\label{theorem01}
Let us assume that the conditions (\ref{>0}) and (\ref{ass}) are met.

Then for all $\beta_1<\beta_2$
there exist positive constants $\mu_{*}$, $\alpha_{*}$, $\delta$, $L$ and $\kappa$
such that for all $(\alpha,\beta,\gamma)$ with
\begin{equation}
\label{par}
\alpha>\alpha_{*},\; \beta_1<\beta<\beta_2 \mbox{ and } 0\le\frac{\gamma}{\alpha} <\mu_*
\end{equation}
the following holds:

(i) The system (\ref{01})--(\ref{02}) has a  three-dimensional integral manifold $\mathfrak{M}(\alpha,\beta,\gamma)$
which can be parametrized by $\psi,\varphi,t \in \mathbb{R}$ in the form
\begin{eqnarray*}
 &  & x=x_{0}(\psi)+\frac{\gamma}{\alpha^{2}}X_{1}\left(\psi,\varphi,\beta t,\alpha t,
\frac{1}{\alpha},\beta,\frac{\gamma}{\alpha}\right)+
\frac{\gamma^{2}}{\alpha^{2}}X_{2}\left(\psi,\varphi,\beta t,\alpha t,\frac{1}{\alpha},
\beta,\frac{\gamma}{\alpha}\right),\label{eq:IM0}\\
 &  & y=r_{0}(\psi)e^{i(\varphi + \phi(\psi))}-i\frac{\gamma}{\alpha}e^{i\alpha t}a(\beta t)
+\frac{\gamma}{\alpha^{2}}Y_{1}\left(\psi,\varphi,\beta t,\alpha t,\frac{1}{\alpha},\beta,\frac{\gamma}{\alpha}\right)
 \nonumber \\
 & & \hspace{10mm}
+ \frac{\gamma^{2}}{\alpha^{2}}Y_{2}\left(\psi,\varphi,\beta t,\alpha t,\frac{1}{\alpha},\beta,\frac{\gamma}{\alpha}\right)
\mbox{ \rm with } \phi(\psi):=\frac{1}{\beta_0}\int_{0}^{\psi}[\Im h(x_0(\xi)) - \alpha_0]d\xi.\label{eq:IM01}
\end{eqnarray*}
Here $X_j: \mathbb{R}^4 \times U \to \mathbb{R}$ and  $Y_j: \mathbb{R}^4 \times U \to \mathbb{C}$
are $C^{l-4}$ smooth,
$4\pi$-periodic with respect to $\psi$ and $2\pi$-periodic with respect to $\varphi,\beta t$ and
$\alpha t$ and
$$
U:=\left\{(\nu,\beta,\mu) \in \mathbb{R}^3:\; 0<\nu<\frac{1}{\alpha_*}, \;\beta_1<\beta<\beta_2, \;0\le \mu<\mu_*\right\}.
$$

(ii) The dynamics of (\ref{01})--(\ref{02}) on  $\mathfrak{M}(\alpha,\beta,\gamma)$ in coordinates $\psi, \varphi$ and $t$ is
determined by a system of the type
\begin{eqnarray}
 & & \frac{d\psi}{dt}=\beta_{0}+\frac{\gamma^{2}}{\alpha^{2}}p^{T}(\psi){\mathcal G}(x_{0}(\psi),
\beta t)+\frac{\gamma^{4}}{\alpha^{4}}\Psi_1\left(\psi,\beta t,\frac{\gamma}{\alpha}\right)
 \nonumber \\
 & &  \hspace{5mm}
  + \frac{\gamma^{2}}{\alpha^{3}}\Psi_2\left(\psi,\varphi,\beta t,\alpha t,
\frac{1}{\alpha},\beta,\frac{\gamma}{\alpha}\right)+\frac{\gamma}{\alpha^{3}}\Psi_3
\left(\psi,\varphi,\beta t,\alpha t,\frac{1}{\alpha},\beta,\frac{\gamma}{\alpha}\right),\label{eq:psi}\\
 & & \frac{d\varphi}{dt}=\alpha_{0}+\frac{\gamma^{2}}{\alpha^{2}}\Phi_1
\left(\psi,\beta t,\frac{\gamma}{\alpha}\right)+\frac{\gamma^{2}}{\alpha^{3}}\Phi_2
\left(\psi,\varphi,\beta t,\alpha t,\frac{1}{\alpha},\beta,\frac{\gamma}{\alpha}\right)\nonumber \\
 & & \hspace{25mm}
 + \frac{\gamma}{\alpha^{3}}\Phi_3\left(\psi,\varphi,\beta t,\alpha t,\frac{1}{\alpha},\beta,
\frac{\gamma}{\alpha}\right),\label{eq:varphi}\end{eqnarray}
where the functions $\Psi_1,\Phi_1: \mathbb{R}^2 \times [0,\mu_*) \to \mathbb{R}$ and  $\Psi_j,\Phi_j: \mathbb{R}^4 \times U \to \mathbb{R}$
$(j=2,3)$ are $C^{l-4}$-smooth, $4\pi$-periodic with respect to $\psi$
and $2\pi$-periodic with respect to $\varphi,\beta t$ and $\alpha t.$

(iii)
The integral manifold  $\mathfrak{M}(\alpha,\beta,\gamma)$ is exponentially attracting (uniformly with respect to
$(\alpha,\beta,\gamma)$ satisfying (\ref{par}))
in the following sense:
For any solution $(x(t),y(t))$ to  (\ref{01})--(\ref{02}) such that
$\mbox{\rm dist}((x(t_{0}),y(t_{0})),\mathcal{T}_{2})<\delta$ for certain
$t_{0} \in \mathbb{R}$ there is a unique solution $(\psi(t),\varphi(t))$ to (\ref{eq:psi})-- (\ref{eq:varphi})
such that
\begin{eqnarray*}
 & & \Bigl\Vert x(t)-x_{0}(\psi(t))-
\frac{\gamma}{\alpha^2}
\tilde X_1(t)
- \frac{\gamma^2}{\alpha^2} \tilde X_2(t) \Bigl\Vert +
 \\
& &
 +\Bigl| y(t)- i\frac\gamma\alpha e^{i\alpha t} a(\beta t)
- r_{0}(\psi(t))e^{i(\varphi(t) + \phi(\psi(t)))}
- \frac{\gamma}{\alpha^2} \tilde Y_{1}(t) -
\frac{\gamma^2}{\alpha^2} \tilde Y_{2}(t)
\Bigl|
\le
 \\
& &
\le
Le^{-\kappa(t-t_{0})}\mbox{\rm dist}\left(\left(x(t_{0}),y(t_{0})\right),
\mathcal{T}_{2}\right), \ t \ge t_0,
\end{eqnarray*}
where
\begin{eqnarray*}
& & \tilde X_j(t) := X_{j}\left(\psi(t),\varphi(t),\beta t,
 \alpha t,\frac{1}{\alpha},\beta,\frac{\gamma}{\alpha}\right),\quad j=1,2, \\
& & \tilde  Y_{j}(t) :=Y_j\left(\psi(t),\varphi(t),\beta t,\alpha t,\frac{1}{\alpha},\beta,
\frac{\gamma}{\alpha}\right), \quad j=1,2.
\end{eqnarray*}
\end{theorem}
~\\

Let us define the function
$$
G(\psi):=\frac{1}{2\pi}\int_{0}^{2\pi}p^{T}(\psi+\theta) 
{\mathcal G}(x_0(\psi+\theta),\psi) d\theta
$$
and the numbers
\[
G_+:=\max_{\psi\in[0,2\pi]}G(\psi),\quad G_-
:=\min_{\psi\in[0,2\pi]}G(\psi).\]
For the sake of simplicity we will suppose that all singular points of $G$ are non-degenerate, i.e.
\begin{equation}
\label{nondeg}
G''(\psi)\not=0 \mbox{ for all } \psi \mbox{ such that } G'(\psi)=0.
\end{equation}
This implies that the set of  singular points of $G$ consists of an even number $2N$ of different points:
$$
\{\psi \in [0,2\pi):\,  G'(\psi)=0\}=\{\psi_1,\ldots,\psi_{2N}\}.
$$
The set of singular values of $G$ will be denoted by
$$
S:=\{G(\psi_1),\ldots,G(\psi_{2N})\}.
$$


The following two theorems describe the dynamics on
$\mathfrak{M}(\alpha,\beta,\gamma)$ in
more details. In particular, they show that for appropriate parameters $(\alpha,\beta,\gamma)$
there appears an even number of two-dimensional integral submanifolds, which determine the
frequency locking behavior we are interested in.

\begin{theorem}
\label{theorem02}
Assume that (\ref{>0}), (\ref{ass}) and (\ref{nondeg}) hold.

Then for any $\varepsilon>0$ there exist positive
$\mu^*$, $\mu_{*}$, and $\delta$ such that for all parameters
$(\alpha,\beta,\gamma)$ satisfying
\begin{equation}
\frac{\mu^*}{\alpha}<\gamma<\mu_{*}\alpha,
\label{eq:cond1}
\end{equation}
\begin{equation}
\label{eq:cond2}
G_-<
\frac{\alpha^2}{\gamma^2}(\beta-\beta_{0})
<G_+,
\end{equation}
\begin{equation}
\mathrm{\rm dist}\left(\frac{\alpha^{2}}{\gamma^{2}}\left(\beta-\beta_{0}\right),S\right)>\varepsilon
\label{eq:cond3}
\end{equation}
the following statements hold:

(i)  The system (\ref{01})--(\ref{02}) has an even number of two-dimensional integral manifolds
$\mathfrak{N}_{j}(\alpha,\beta,\gamma)
\subset \mathfrak{M}(\alpha,\beta,\gamma), \;j=1,...,2\hat{N}(\alpha,\beta,\gamma), \;0<\hat{N}(\alpha,\beta,\gamma)\le N$
which can be parametrized by $\varphi, t \in \mathbb{R}$ in the form
\begin{eqnarray*}
x & = & x_{0}(\beta t+\vartheta_{j})+\frac{\gamma}{\alpha}X_{1j}
\left(\varphi,\beta t,\alpha t,\frac{1}{\alpha},\beta,\frac{\gamma}{\alpha}\right)+
\frac{1}{\alpha}X_{2j}\left(\varphi,\beta t,\alpha t,\frac{1}{\alpha},\beta,\frac{\gamma}{\alpha}\right),\label{xm}\\[2mm]
y & = &
r_{0}(\beta t+\vartheta_{j})e^{i\varphi + \phi(\beta t+\vartheta_{j})} \\[2mm]
& & \hspace{15mm} +
\frac{\gamma}{\alpha}Y_{1j}\left(\varphi,\beta t,\alpha t,\frac{1}{\alpha},\beta,
\frac{\gamma}{\alpha}\right)+\frac{1}{\alpha}Y_{2j}\left(\varphi,\beta t,\alpha t,
\frac{1}{\alpha},\beta,\frac{\gamma}{\alpha}\right),\label{ym}
\end{eqnarray*}
where $\vartheta_{j}$ are constants, and the functions
$X_{kj},Y_{kj}: \mathbb{R}^4 \times V \to \mathbb{R}$
are $C^{l-4}$-smooth and
$2\pi$-periodic with respect to $\varphi,\beta t$ and
$\alpha t$, and
$$
V:=\left\{(\nu,\beta,\mu):G_-<\mu^2(\beta-\beta_0)<G_+,
\mu^*\nu^2< \mu<\mu_*, \mbox{\rm dist}(\mu^2(\beta-\beta_{0}),S)>\varepsilon
\right\}.
$$

(ii) The dynamics of (\ref{01})--(\ref{02}) on  $\mathfrak{N}_j(\alpha,\beta,\gamma)$ in coordinates $\varphi$ and $t$ is
determined by an equation of the type
\begin{eqnarray*}
\frac{d\varphi}{dt}=\alpha_{0}
+\frac{\gamma^{2}}{\alpha^{3}}
\Phi_{1j}\left(\varphi,\beta t,\alpha t,\frac{1}{\alpha},\beta,\frac{\gamma}{\alpha}\right)+\frac{\gamma}{\alpha^{3}}\Phi_{2j}
\left(\varphi,\beta t,\alpha t,\frac{1}{\alpha},\beta,\frac{\gamma}{\alpha}\right),
\end{eqnarray*}
where the functions $\Phi_{kj}:\mathbb{R}^3 \times V \to \mathbb{R}$ are $C^{l-4}$ smooth
and $2\pi$-periodic in $\varphi,\beta t$ and $\alpha t$.

(iii) Any solution $(x(t),y(t))$ to  (\ref{01})--(\ref{02}) such that
$\mbox{\rm dist}((x(t_{0}),y(t_{0})),\mathcal{T}_{2})<\delta$ for certain
$t_{0} \in \mathbb{R}$ tends to one of the manifolds $\mathfrak{N}_{j}(\alpha,\beta,\gamma)$
as $t\to \infty.$
\end{theorem}



\begin{theorem}
\label{thm:main}
Assume that (\ref{>0}), (\ref{ass}) and (\ref{nondeg}) hold.

Then for any $\varepsilon>0$ and $\varepsilon_{1}>0$
there exist positive $\mu^{*}$, $\mu_{*}$ and $\delta$
such that for all parameters $(\alpha,\beta,\gamma)$ satisfying the
conditions (\ref{eq:cond1})--(\ref{eq:cond3})
and for any solution $(x(t),y(t))$ of system (\ref{01})--(\ref{02})
 such that
$\mbox{\rm dist}((x(t_{0}),y(t_{0})),\mathcal{T}_{2})<\delta$ for certain
$t_{0} \in \mathbb{R}$ there exist $\sigma, T \in \mathbb{R}$ such that
$$
\| x(t)-x_{0}(\beta t+\sigma)\|+\bigl||y(t)|-|y_{0}(\beta t+\sigma)|\bigr|<\varepsilon_{1}
\mbox{ for all } t>T.
$$
\end{theorem}
\vspace{3mm}

The conditions (\ref{eq:cond1})--(\ref{eq:cond3}) from Theorems \ref{theorem02}
and \ref{thm:main} determine the so-called locking region, i.e. the set of all triples
$(\alpha,\beta,\gamma)$ for which  modulation frequency locking
takes place. These domains are illustrated in the figures \ref{fig:Graphs}--\ref{fig:crossing}.

\begin{figure}
\begin{center}
\includegraphics[width=0.85\linewidth]{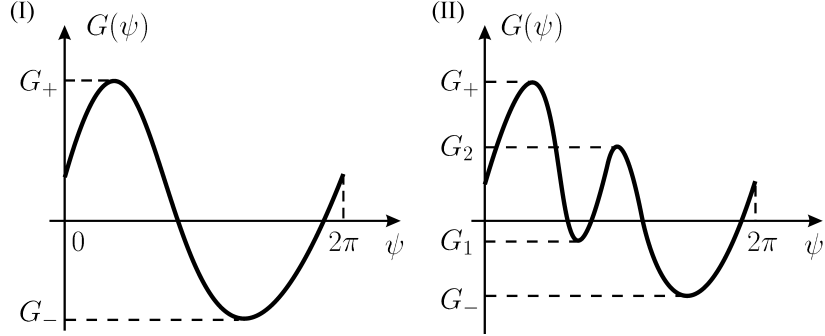}
\end{center}
\caption{\label{fig:Graphs}
Graphs of the function $G$.
}
\end{figure}

In Fig.~\ref{fig:Graphs} we show two typical cases of graphs of the function $G$. In the case (I) there exist one
positive and one
negative local
extremum and in the case (II) two positive and two negative local extrema, i.e.,
\begin{eqnarray*}
\mbox{(I) }: && \ N=1, \; S=\{G_-,G_+\}, \; G_-<0<G_+,\\
\mbox{(II)}: && \ N=2, \; S=\{G_-,G_1,G_2,G_+\}, \; G_-<G_1<0<G_2<G_+.
\end{eqnarray*}

\begin{figure}
\begin{center}
\includegraphics[width=0.85\linewidth]{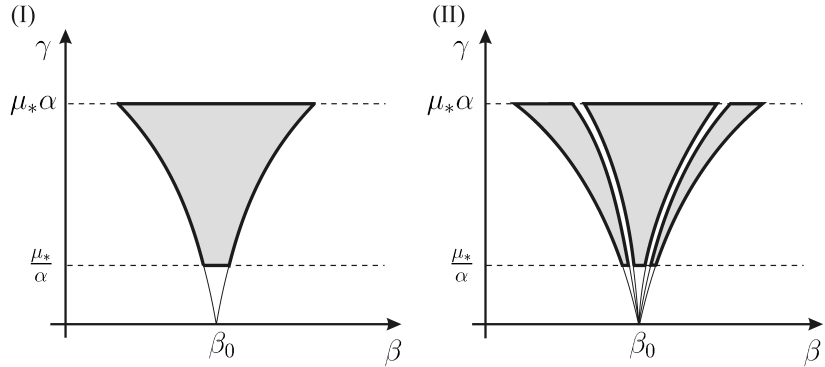}
\end{center}
\caption{\label{fig:alpha=const}
Cross-sections of the locking region $\alpha={\mathrm const}$.
}
\end{figure}

In Fig.~\ref{fig:alpha=const} we show $\alpha$=const sections of the locking region.
In the case (I) this  section is
$$
\left\{(\beta,\gamma):\;\frac{\mu^*}{\alpha}<\gamma<\mu_* \alpha, \;
G_-+\varepsilon<
\frac{\alpha^2}{\gamma^2}(\beta-\beta_{0})
<G_+-\varepsilon\right\}.
$$
It is bounded by two straight lines $\gamma=\mu^*/\alpha$ and  $\gamma=\mu_* \alpha$ and by two square root like curves
$$
\gamma=\alpha \sqrt{\frac{\beta-\beta_0}{\tilde{G}}} \mbox{ with } \tilde{G} \in \{G_-+\varepsilon,G_+ -\varepsilon\}.
$$
In the case (II) the $\alpha$=const  section is bounded by the same two horizontal  straight lines and by six  square root like curves
$$
\gamma=\alpha \sqrt{\frac{\beta-\beta_0}{\tilde{G}}} \mbox{ with } \tilde{G} \in \{G_-+\varepsilon,
G_1-\varepsilon, G_1+\varepsilon,G_2-\varepsilon,G_2+\varepsilon,G_+-\varepsilon\}.
$$

\begin{figure}
\begin{center}
\includegraphics[width=0.85\linewidth]{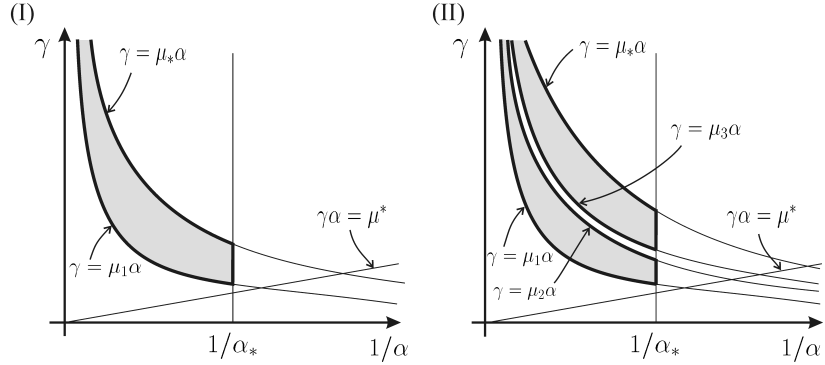}
\end{center}
\caption{\label{fig:beta=const}
Cross-sections of the locking region $\beta={\mathrm{const}}$.
}
\end{figure}

Finally, in Fig.~\ref{fig:beta=const} we show $\beta$=const sections of the locking region in the
$(1/\alpha,\gamma)$ plane.
We consider the parameter $\alpha$ in the region $\alpha > \alpha_*$ with sufficiently large  $\alpha_*>0$
\begin{equation}
\label{1new}
\alpha_*^2>\frac{\mu^*}{\mu_*}\sqrt{\frac{G_+-\varepsilon}{G_2-\varepsilon}}.
\end{equation}
If (\ref{1new}) is satisfied, consider the set of all $\beta > \beta_0$ such that
\begin{equation}
\label{2new}
\frac{\mu^*}{\alpha_*}<\alpha_*\sqrt{\frac{\beta-\beta_0}{G_+-\varepsilon}} \mbox{ and }
\sqrt{\frac{\beta-\beta_0}{G_2-\varepsilon}}<\mu_*.
\end{equation}
For any fixed $\alpha > \alpha_*$, where $\alpha_*$ satisfies (\ref{1new}),
and for  any fixed $\beta > \beta_0$
with (\ref{2new}),
the line $\{(\alpha,\beta,\gamma): \gamma \in \mathbb{R}\}$ crosses the boundary
of the locking region in two points $\gamma=\mu_1 \alpha$ and $\gamma=\mu_* \alpha$ in case (I) and in four  points
$\gamma=\mu_1 \alpha$, $\gamma=\mu_2 \alpha$, $\gamma=\mu_3 \alpha$   and $\gamma=\mu_* \alpha$  in case (II)
(see also Fig.~\ref{fig:crossing}).
Here we denoted
$$
\mu_1=\sqrt{\frac{\beta-\beta_0}{G_+-\varepsilon}},\;
\mu_2=\sqrt{\frac{\beta-\beta_0}{G_2+\varepsilon}},\;
\mu_3=\sqrt{\frac{\beta-\beta_0}{G_2-\varepsilon}}.
$$

\begin{figure}
\begin{center}
\includegraphics[width=0.85\linewidth]{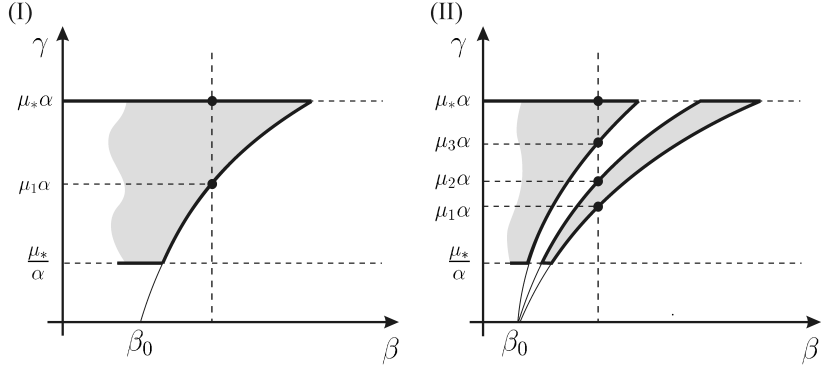}
\end{center}
\caption{\label{fig:crossing}
Intersection of a line $\beta={\mathrm{const}}$ with the boundary of the
locking region.
}
\end{figure}


\section{\textbf{Averaging\label{sec:Averaging}} }

In this section we perform changes of variables with the aim to average
the nonautonomous terms with fast oscillating arguments $\alpha t$.
 As the result of these transformations, we obtain an equivalent
system, where the fast oscillating terms have the order of magnitude
of $\gamma^{2}/\alpha^{2}$ and smaller. The principles and details
of the averaging procedure can be found e.g. in \cite{Bogoliubov1961}.

Performing the change of variables \begin{eqnarray}
x & = & x_{1},\nonumber \\
y & = & y_{1}-i\frac{\gamma}{\alpha}e^{i\alpha t}a(\beta t) \nonumber
\end{eqnarray}
in (\ref{01})--(\ref{02}), we obtain the transformed system
\begin{eqnarray}
& & \frac{dx_{1}}{dt} = f(x_{1})+g(x_{1})|y_{1}|^{2}+\frac{\gamma^{2}}{\alpha^{2}}g(x_{1})|a(\beta t)|^{2}-\frac{2\gamma}{\alpha}g(x_{1})\Im\{y_{1}e^{-i\alpha t}a^{*}(\beta t)\}, \quad
\label{averaging000}\\
& & \frac{dy_{1}}{dt} = h(x_{1})y_{1}-i\frac{\gamma}{\alpha}e^{i\alpha t}\left(h(x_{1})a(\beta t)-\beta\frac{da}{dt}(\beta t)\right),
\label{averaging001}
\end{eqnarray}
where $*$ denotes complex conjugation. In system (\ref{averaging000})--(\ref{averaging001}),
the fast oscillatory terms with frequency $\alpha t$ are now proportional to
$\gamma/\alpha$. Since the first averaging has not produced any nontrivial
contributions on the zeroth order, the second averaging transformation is necessary:

\begin{eqnarray*}
& & x_{1}  =  x_{2}-2\frac{\gamma}{\alpha^{2}}g(x_{1})\Re\{y_{1}e^{-i\alpha t}a^{*}(\beta t)\},\\
& & y_{1}  =  y_{2}-\frac{\gamma}{\alpha^{2}}e^{i\alpha t}\left(h(x_{1})a(\beta t)-\beta\frac{da}{dt}(\beta t)\right),
\end{eqnarray*}
which  allows eliminating fast oscillating terms of order $\gamma/\alpha$.
 \begin{eqnarray*}
\frac{dx_{2}}{dt} &  = & f(x_{2})+g(x_{2})|y_{2}|^{2}+\frac{\gamma^{2}}{\alpha^{2}}g(x_{2})|a(\beta t)|^{2}\\
 & + & 2\frac{\gamma}{\alpha^{2}}\left(\frac{dg(x_{2})}{dx_{2}}f(x_{2})-\frac{df(x_{2})}{dx_{2}}g(x_{2})\right)\Re\{y_{2}e^{-i\alpha t}a^{*}(\beta t)\}\\
 & + & 2\frac{\gamma}{\alpha^{2}}g(x_{2})\Re\left\{ e^{i\alpha t}(h^{*}(x_{2})a(\beta t)+2\beta\frac{da}{dt}(\beta t)-h(x_{2})a(\beta t))\right\} \\
 & + & \frac{\gamma^{2}}{\alpha^{3}}r_{1}(x_{2},y_{2},\alpha t,\beta t,\frac{\gamma}{\alpha},\frac{1}{\alpha}),\\[3mm]
\frac{dy_{2}}{dt} & = & h(x_{2})y_{2}+ \frac{\gamma}{\alpha^{2}}e^{i\alpha t}\biggl( 2\beta h(x_{2})\frac{da}{dt}(\beta t)-h^2(x_2)a(\beta t)\\
 & -  & \beta^{2}\frac{d^{2}a}{dt^{2}}(\beta t) + \frac{dh(x_{2})}{dx_{2}}(f(x_{2})+g(x_{2})|y_{2}|^{2})a(\beta t)\biggl) \\
 & - & 2\frac{\gamma}{\alpha^{2}}g(x_{2})\frac{dh(x_{2})}{dx_{2}}\Re\{y_{2}e^{-i\alpha t}a^{*}(\beta t)\}
+ \frac{\gamma^{2}}{\alpha^{3}}r_{2}(x_{2},y_{2},e^{i\alpha t},\beta t,\frac{\gamma}{\alpha},\frac{1}{\alpha}),\\
\end{eqnarray*}
where the remainder terms $r_{1},r_{2}$ are $C^{l-2}$ smooth functions
in all arguments and $2\pi$-periodic in $\alpha t$ and in $\beta t$.

Again, the
second transformation has not produced any nontrivial contributions of the order
$1/\alpha$. Let us perform the third change of variables \begin{eqnarray*}
x_{2} & = & x_{3}-2\frac{\gamma}{\alpha^{3}}\left(\frac{dg(x_{2})}{dx_{2}}f(x_{2})-\frac{df(x_{2})}{dx_{2}}g(x_{2})\right)\Im\{y_{2}e^{-i\alpha t}a^{*}(\beta t)\}\\
 & + & 2\frac{\gamma}{\alpha^{3}}g(x_{2})\Im\{e^{i\alpha t}(h^{*}(x_{2})a(\beta t)+2\beta\frac{da}{dt}(\beta t)-h(x_{2})a(\beta t))\},\\[3mm]
y_{2} & = & y_{3}-i\frac{\gamma}{\alpha^{3}}e^{i\alpha t}\biggl(2h(x_{2})\beta\frac{da}{dt}(\beta t)-h^2(x_{2})a(\beta t)-\beta^{2}\frac{d^{2}a}{dt^{2}}(\beta t)\\
 & + & \frac{dh(x_{2})}{dx_{2}}(f(x_{2})+g(x_{2})|y_{2}|^{2})a(\beta t)\biggr)
  +  2\frac{\gamma}{\alpha^{3}}\frac{dh(x_{2})}{dx_{2}}g(x_{2})\Im\{y_{2}e^{-i\alpha t}a^{*}(\beta t)\},\end{eqnarray*}
which transforms the system to the following form:
\begin{eqnarray}
\frac{dx_{3}}{dt} & = & f(x_{3})+g(x_{3})|y_{3}|^{2}+\frac{\gamma^{2}}{\alpha^{2}}g(x_{3})|a(\beta t)|^{2}\nonumber \\
 & + & \frac{\gamma}{\alpha^{3}}r_{3}\left(x_{3},y_{3},\alpha t,\beta t,\frac{\gamma}{\alpha},\frac{1}{\alpha}\right)+\frac{\gamma^{2}}{\alpha^{3}}r_{4}\left(x_{3},y_{3},\alpha t,\beta t,\frac{\gamma}{\alpha},\frac{1}{\alpha}\right),\label{eq:AV}\\[3mm]
\frac{dy_{3}}{dt} & = & h(x_{3})y_{3}\nonumber \\
 & + & \frac{\gamma}{\alpha^{3}}r_{5}\left(x_{3},y_{3},\alpha t,\beta t,\frac{\gamma}{\alpha},\frac{1}{\alpha}\right)+\frac{\gamma^{2}}{\alpha^{3}}r_{6}\left(x_{3},y_{3},\alpha t,\beta t,\frac{\gamma}{\alpha},\frac{1}{\alpha}\right),\label{eq:AV1}
\end{eqnarray}
where the remainder terms $r_{3},...,r_{6}$ are $2\pi$-periodic in $\alpha t$
and in $\beta t$, of class $C^{l-3}$ in all variables.
The obtained system (\ref{eq:AV})--(\ref{eq:AV1}) contains a nontrivial contribution
of the order $\gamma^2/\alpha^2$ and all fast oscillatory terms of the orders
$\gamma/\alpha^3$,  $\gamma^2/\alpha^3$ and smaller. The next section proceeds with the analysis
of the averaged system (\ref{eq:AV})--(\ref{eq:AV1}).

\section{\textbf{Local coordinates.} \label{sec:Local-coordinates}}

Let us introduce two new parameters
$$
\mu:=\frac{\gamma}{\alpha},\quad\varepsilon:=\frac{1}{\alpha}.
$$
We assume that $\mu\in(0,\mu_{0})$ and $\varepsilon\in(0,\varepsilon_{0})$
with some sufficiently small $\mu_{0}>0,$ $\varepsilon_{0}>0$. The
system (\ref{eq:AV})--(\ref{eq:AV1}) can be re-written as \begin{eqnarray}
\frac{dx_{3}}{dt} & = & f(x_{3})+g(x_{3})|y_{3}|^{2}+\mu^{2}g(x_{3})|a(\beta t)|^{2}\label{1}\\[2mm]
 &  & \qquad+\varepsilon^{2}\mu r_{3}(x_{3},y_{3},\beta t,\alpha t,\mu,\varepsilon)+\varepsilon\mu^{2}r_{4}(x_{3},y_{3},\beta t,\alpha t,\mu,\varepsilon),\nonumber \\
\frac{dy_{3}}{dt} & = & h(x_{3})y_{3}+\varepsilon^{2}\mu r_{5}(x_{3},y_{3},\beta t,\alpha t,\mu,\varepsilon)+\varepsilon\mu^{2}r_{6}(x_{3},y_{3},\beta t,\alpha t,\mu,\varepsilon).\label{2}\end{eqnarray}

After the change of variables \begin{equation}
y_{3}=re^{i\theta},\label{eq:polar}\end{equation}
in polar coordinates $(r,\theta)$ the system (\ref{1})--(\ref{2})
takes the form \begin{eqnarray}
\frac{dx_{3}}{dt} & = & f(x_{3})+g(x_{3})r^{2}+\mu^{2}g(x_{3})|a(\beta t)|^{2}+\varepsilon\mu^{2}f_{1}+\varepsilon^{2}\mu f_{2}, \label{pc1}\\[2mm]
\frac{dr}{dt} & = & \Re h(x_{3})r+\varepsilon\mu^{2}f_{3}+\varepsilon^{2}\mu f_{4},\label{pc2}\\[2mm]
\frac{d\theta}{dt} & = & \Im h(x_{3})+\varepsilon\mu^{2}f_{5}+\varepsilon^{2}\mu f_{6},\label{pc3}\end{eqnarray}
where $f_{j}=f_{j}(x_{3},r,\theta,\beta t,\alpha t,\mu,\varepsilon)$,
$j=1,\dots,6,$ are $C^{l-3}$-smooth and $2\pi$-periodic in $\theta,\beta t$,
$\alpha t$ functions. Here we assume $r\ge r_{*}=\frac{1}{2}\min_{\psi}|y_{0}(\psi)|>0$.

By substituting $z=(x_{3},r),$ system (\ref{pc1}) - (\ref{pc3})
takes the following form
\begin{eqnarray}
 &  & \frac{dz}{dt}=F(z)+\mu^{2}{\mathcal G}(z,\beta t)+\varepsilon\mu^{2}F_{1}+\varepsilon^{2}\mu F_{2},\label{f1}\\[3mm]
 &  & \frac{d\theta}{dt}=h_{2}(z)+\varepsilon\mu^{2}F_{3}+\varepsilon^{2}\mu F_{4},\label{f2}
\end{eqnarray}
 where the functions $h_2$, $F$, and $\mathcal{G}$ are defined by
 $h_{2}(z):=\Im h(x_{3})$,
\[
F(z):=\left[\begin{array}{c}
f(x_{3})+g(x_{3})r^{2}\\
\Re h(x_{3})r\end{array}\right],\quad \mathcal{G}(z,\beta t):=\left[\begin{array}{c}
g(x_{3})|a(\beta t)|^{2}\\
0\end{array}\right],\]
$F_{j}=F_{j}(z,\theta,\beta t,\alpha t,\mu,\varepsilon)$, $j=1,\dots,4$
are $C^{l-3}$-smooth and $2\pi$-periodic in $\theta,\beta t$ and
$\alpha t$ functions. The above defined function $\mathcal{G}$,
which is defined for $z=(r,x_3)\in\mathbb{R}^{n+1}$, on the subspace
$r=0$, i.e. of all vectors $(0,x_3)$, is just the function $\mathcal{G}$
defined in (\ref{mathcalG}). Therefore, the use of the same notations should not
lead to misunderstanding.

Equation
$$
\frac{dz}{dt}  =  F(z)
$$
has the periodic solution $z(t)=z_{0}(\beta_{0}t)=(x_{0}(\beta_{0}t),r_{0}(\beta_{0}t))$
and the corresponding limit cycle in $\mathbb{R}^{n+1}$ is $z=z_{0}(\psi),\ \psi\in\mathbb{T}_{1},$
i.e.,
\begin{equation} \label{411}
\frac{dz_{0}(\psi)}{d\psi}=\frac{F(z_{0}(\psi))}{\beta_{0}},\ \psi\in\mathbb{T}_{1}.
\end{equation}
Let $\Omega_{1}(\psi)$ be the fundamental matrix solution for the
variational equation \begin{eqnarray}
\frac{d\delta z}{d\psi} & = & \frac{1}{\beta_{0}}\frac{\partial F(z_{0}(\psi))}{\partial z}\delta z.\label{variational}\end{eqnarray}
along the periodic solution $z_{0}(\psi)$.

By the Floquet theorem, the fundamental matrix $\Omega_{1}(\psi)$
can be represented in the form \begin{eqnarray}
\Omega_{1}(\psi) & = & \Phi_{1}(\psi)e^{H_{1}\psi/\beta_{0}},\label{eq:Omega1}\end{eqnarray}
where $\Phi_{1}(\psi)$ is $4\pi$-periodic $(n+1)\times(n+1)$ real
matrix and $H_{1}$ is $(n+1)\times(n+1)$ constant real matrix.

Since $dz_{0}(\psi)/d\psi$ is a periodic solution of (\ref{variational}),
we can choose \[
\Omega_{1}(\psi)=\left[\frac{dz_{0}(\psi)}{d\psi},\Omega(\psi)\right],\quad\Phi_{1}(\psi)=\left[\frac{dz_{0}(\psi)}{d\psi},\Phi(\psi)\right],\]
where $\Omega(\psi)$ and $\Phi(\psi)$ are $(n+1)\times n$ matrices
and $H_{1}=\mathrm{diag}\{0,H\}$ with $n\times n$ constant matrix
$H.$ Since the periodic solution $z_{0}(\psi)$ is orbitally stable,
all eigenvalues of matrix $H$ have negative real parts.

Let us find the inverse matrix for $\Phi_{1}(\psi):$ \[
\Phi_{1}^{-1}(\psi)=\left(\Omega_{1}(\psi)e^{-H_{1}\psi/\beta_{0}}\right)^{-1}=e^{H_{1}\psi/\beta_{0}}\Omega_{1}^{-1}(\psi).\]

Taking into account that \begin{eqnarray}
\tilde{\Omega}_{1}^{T}(\psi)\Omega_{1}(\psi)=I,\quad\psi\in\mathbb{R},\label{product}\end{eqnarray}
where $I$ is the identity matrix and $\tilde{\Omega}_{1}(\psi)$
is the fundamental matrix solution of the adjoint system \begin{eqnarray}
\frac{dw}{d\psi}=-\left(\frac{1}{\beta_{0}}\frac{\partial F(z_{0}(\psi))}{\partial\psi}\right)^{T}w,\label{adjoint}\end{eqnarray}
we conclude that $\Omega_{1}^{-1}(\psi)=\tilde{\Omega}_{1}^{T}(\psi)$
(see \cite{Demidovich1967}). Accordingly to Floquet theorem \[
\tilde{\Omega}_{1}(\psi)=\tilde{\Phi}_{1}(\psi)e^{\tilde{H}_{1}\psi/\beta_{0}}.\]
 It follows from (\ref{eq:Omega1}) and (\ref{product}) that \[
\tilde{\Omega}_{1}(\psi)=\left(\Omega_{1}^{-1}(\psi)\right)^{T}=\left(\Phi_{1}^{-1}(\psi)\right)^{T}e^{-H_{1}^{T}\psi/\beta_{0}}.\]
 Hence
$$
\tilde{\Phi}_{1}^{T}(\psi)=\Phi_{1}^{-1}(\psi),\quad\tilde{H}_{1}^{T}=-H_{1}.
$$

Since the linear periodic system (\ref{variational}) has one nonzero
linearly independent periodic solution, the adjoint system (\ref{adjoint})
has also one nonzero linearly independent periodic solution. Then
\[
\tilde{\Omega}_{1}(\psi)=\left[p(\psi),\tilde{\Omega}(\psi)\right],\quad\tilde{\Phi}_{1}(\psi)=\left[p(\psi),\tilde{\Phi}(\psi)\right],\]
 where $p(\psi)$ is $2\pi$-periodic solution of adjoint system (\ref{adjoint})
and $\tilde{\Omega}(\psi)$ and $\tilde{\Phi}(\psi)$ are $(n+1)\times n$
matrices, $\tilde{\Phi}(\psi)$ is $4\pi$ periodic. Taking into account
(\ref{product}), we obtain that the scalar product in $\mathbb{R}^{n+1}$
of two vectors $dz_{0}(\psi)/d\psi$ and $p(\psi)$ is equal to $1$
for all $\psi\in\mathbb{T}_{1}.$

It can be verified that \[
\beta_{0}\frac{d\Phi_{1}(\psi)}{d\psi}+\Phi_{1}(\psi)H_{1}=\frac{\partial F(z_{0}(\psi))}{\partial\psi}\Phi_{1}(\psi).\]
 Then $(n+1)\times n$-matrix $\Phi(\psi)$ satisfies relation \begin{eqnarray}
\beta_{0}\frac{d\Phi(\psi)}{d\psi}+\Phi(\psi)H & = & \frac{\partial F(z_{0}(\psi))}{\partial\psi}\Phi(\psi).\label{relation1}\end{eqnarray}

We introduce new coordinates $\psi$ and $h$ instead of $z$ in the
neighborhood of the periodic solution $z_{0}$ by the formula \begin{equation}
z=z_{0}(\psi)+\Phi(\psi)h,\label{eq:zamina}\end{equation}
 where $h\in\mathbb{R}^{n}$, $\|h\|\le h_{0}$ with some $h_{0}>0$.
After substituting (\ref{eq:zamina}) into (\ref{f1}) we obtain \begin{eqnarray}
 &  & \left(\frac{dz_{0}(\psi)}{d\psi}+\frac{d\Phi(\psi)}{d\psi}h\right)\frac{d\psi}{dt}+\Phi(\psi)\frac{dh}{dt}\nonumber \\[3mm]
 &  & =F(z_{0}(\psi)+\Phi(\psi)h)+\mu^{2}\mathcal{G}(z_{0}(\psi)+\Phi(\psi)h,\beta t)\nonumber \\[3mm]
 &  & +\varepsilon\mu^{2}F_{1}(z_{0}(\psi)+\Phi(\psi)h,\theta,\beta t,\alpha t,\mu,\varepsilon))
 \nonumber \\[3mm]
 &  & +
 \varepsilon^{2}\mu F_{2}(z_{0}(\psi)+\Phi(\psi)h,\theta,\beta t,\alpha t,\mu,\varepsilon)).\label{relation2}\end{eqnarray}
 With regard for (\ref{411}) and (\ref{relation1}),
the relation (\ref{relation2}) yields \begin{eqnarray}
 &  & \left(\frac{dz_{0}(\psi)}{d\psi}+\frac{d\Phi(\psi)}{d\psi}h\right)\left(\frac{d\psi}{dt}-\beta_{0}\right)+\Phi(\psi)\left(\frac{dh}{dt}-Hh\right)\nonumber \\[3mm]
 &  & =F(z_{0}(\psi)+\Phi(\psi)h)-F(z_{0}(\psi))-\frac{\partial F(z_{0}(\psi))}{\partial\psi}\Phi(\psi)h\nonumber \\[3mm]
 &  & +\mu^{2}\mathcal{G}(z_{0}(\psi)+\Phi(\psi)h,\beta t)+\varepsilon\mu^{2}F_{1}(z_{0}(\psi)+\Phi(\psi)h,\theta,\beta t,\alpha t,\mu,\varepsilon) \nonumber \\[3mm]
 &  & +\varepsilon^{2}\mu F_{2}(z_{0}(\psi)+\Phi(\psi)h,\theta,\beta t,\alpha t,\mu,\varepsilon). \label{relation3} \end{eqnarray}

Since by our construction $\det\left[\frac{dz_{0}(\psi)}{d\psi},\Phi(\psi)\right]=\det\Phi_{1}(\psi)\neq0$
for all $\psi,$ the matrix \[
\left[\frac{dz_{0}(\psi)}{d\psi}+\frac{d\Phi(\psi)}{d\psi}h,\Phi(\psi)\right]\]
 is invertible for sufficiently small $h$. Therefore taking into
account the expansion \[
(A+B)^{-1}=A^{-1}-A^{-1}BA^{-1}+A^{-1}BA^{-1}BA^{-1}-...,\]
 we obtain for sufficiently small $h:$ \begin{eqnarray*}
 &  & \left[\frac{dz_{0}(\psi)}{d\psi}+\frac{d\Phi(\psi)}{d\psi}h,\Phi(\psi)\right]^{-1}=\left[\Phi_{1}(\psi)+\left[\frac{d\Phi(\psi)}{d\psi}h,0\right]\right]^{-1}\\[2mm]
 &  & =\Phi_{1}^{-1}(\psi)+\tilde{H}(h,\psi,\mu)=\tilde{\Phi}_{1}^{T}(\psi)+\tilde{H}(h,\psi,\mu)\\
 &  & =\left[\begin{array}{c}
p^{T}(\psi)\\
\tilde{\Phi}^{T}(\psi)\end{array}\right]+\left[\begin{array}{c}
\tilde{H}_{1}(h,\psi,\mu)\\
\tilde{H}_{2}(h,\psi,\mu)\end{array}\right],\end{eqnarray*}
 where the $C^{l-4}$-smooth function $\tilde{H}(h,\psi,\mu)=\mathcal{O}(\|h\|)$
is periodic in $\psi$.

Hence, the equation (\ref{relation3}) can be solved with respect
to the derivatives $d\psi/dt$ and $dh/dt:$
\begin{eqnarray}
 &  & \frac{dh}{dt}=Hh+\mu^{2}[\tilde{\Phi}^{T}(\psi)+\tilde{H}_{2}(h,\psi,\mu)]\mathcal{G}(z_{0}(\psi)+\Phi(\psi)h,\beta t)\nonumber \\
 &  & \qquad\qquad+[\tilde{\Phi}^{T}(\psi)+\tilde{H}_{2}(h,\psi,\mu)][F_{5}+\varepsilon\mu^{2}F_{1}+\varepsilon^{2}\mu F_{2}],\label{s1}\\
 &  & \frac{d\psi}{dt}=\beta_{0}+\mu^{2}p^{T}(\psi)\mathcal{G}(z_{0}(\psi),\beta t)+\mu^{2}\tilde{H}_{1}(h,\psi,\mu)\mathcal{G}(z_{0}(\psi),\beta t)\nonumber \\
 &  & \qquad\qquad+[p^{T}(\psi)+\tilde{H}_{1}(h,\psi,\mu)][\mu^{2}G_{1}+F_{5}+\varepsilon\mu^{2}F_{1}+\varepsilon^{2}\mu F_{2}], \label{s2}
\end{eqnarray}
 where \[
F_{5}(h,\psi,\mu)=F(z_{0}(\psi)+\Phi(\psi)h)-F(z_{0}(\psi))-\frac{\partial F(z_{0}(\psi))}{\partial\psi}\Phi(\psi)h,\]
\[
G_{1}(h,\psi,\beta t,\mu)=\mathcal{G}(z_{0}(\psi)+\Phi(\psi)h,\beta t)-\mathcal{G}(z_{0}(\psi),\beta t).\]
\[
F_{j}=F_{j}(z_{0}+\Phi(\psi)h,\theta,\beta t,\alpha t,\mu,\varepsilon),\quad j=1,2.\]

We supplement this system with equation (\ref{f2}): \begin{eqnarray}
\frac{d\theta}{dt} & = & h_{2}(z_{0}(\psi)+\Phi(\psi)h)+\varepsilon\mu^{2}F_{3}(z_{0}+\Phi(\psi)h,\theta,\beta t,\alpha t,\mu,\varepsilon)\nonumber \\
 &  & +\varepsilon^{2}\mu F_{4}(z_{0}+\Phi(\psi)h,\theta,\beta t,\alpha t,\mu,\varepsilon).\label{s3}\end{eqnarray}
 Using the equality \[
\frac{1}{2\pi}\int_{0}^{2\pi}h_{2}(z_{0}(\psi))d\psi=\alpha_{0},\]
 we replace the angular variable $\theta$ in system (\ref{s1}) --
(\ref{s3}) by $\varphi$ accordingly to the formula \[
\theta=\varphi+\frac{1}{\beta_{0}}\int^{\psi}[h_{2}(z_{0}(\xi))-\alpha_{0}]d\xi,\]
 where $\int^{\psi}$ is a certain antiderivative of the function
$h_{2}(z_{0}(\xi))-\alpha_{0}.$

As a result we obtain the following system \begin{eqnarray}
 &  & \frac{dh}{dt}=Hh+\mu^{2}R_{11}+R_{12}+\varepsilon\mu^{2}R_{13}+
\varepsilon^{2}\mu R_{14},\label{ss1}\\[3mm]
 &  & \frac{d\psi}{dt}=\beta_{0}+\mu^{2}p^{T}(\psi)\mathcal{G}(z_{0}(\psi),\beta t)
+\mu^{2}R_{21}+R_{22}+\varepsilon\mu^{2}R_{23}+\varepsilon^{2}\mu R_{24},
\label{ss2}\\[2mm]
 &  & \frac{d\varphi}{dt}=\alpha_{0}+\mu^{2}R_{31}+R_{32}+
\varepsilon\mu^{2}R_{33}+\varepsilon^{2}\mu R_{34},\label{ss3}\end{eqnarray}
where functions
\begin{eqnarray*}
R_{11} & = & R_{11}(h,\psi,\beta t,\mu)=[\tilde{\Phi}^{T}(\psi)+\tilde{H}_{2}(h,\psi,\mu)]\mathcal{G}(z_{0}(\psi)+\Phi(\psi)h,\beta t), \\[2mm]
R_{12} & = & R_{12}(h,\psi,\beta t,\mu)=[\tilde{\Phi}^{T}(\psi)+\tilde{H}_{2}(h,\psi,\mu)]F_{5}=\mathcal{O}(\|h\|^{2}), \\[2mm]
R_{21} & = & R_{21}(h,\psi,\beta t,\mu)=\tilde{H}_{1}(h,\psi,\mu)\mathcal{G}(z_{0}(\psi),\beta t), \\[2mm]
& & \hspace{25mm} + [p^{T}(\psi)+\tilde{H}_{1}(h,\psi,\mu)]G_{1}(h,\psi,\beta t,\mu)=\mathcal{O}(\|h\|) \\[2mm]
R_{22} & = &  R_{22}(h,\psi,\beta t,\mu)=[p^{T}(\psi)+\tilde{H}_{1}(h,\psi,\mu)]F_{5}(h,\psi,\beta t,\mu)=
\mathcal{O}(\|h\|^{2}), \\[2mm]
R_{31} & = & R_{31}(h,\psi,\beta t,\mu)=\frac{1}{\beta_{0}}[\alpha_{0}-
h_{2}(z_{0}(\psi))]\left(p^{T}(\psi)\mathcal{G}(z_{0}(\psi),\beta t)+R_{21}\right), \\[2mm]
R_{32} & = & R_{32}(h,\psi,\beta t,\mu)=h_{2}(z_{0}(\psi)+\Phi(\psi)h)-h_{2}(z_{0}(\psi)) \\[2mm]
& & \hspace{25mm} - \frac{1}{\beta_{0}}[h_{2}(z_{0}(\psi))-\alpha_{0}]R_{22}=\mathcal{O}(\|h\|)
\end{eqnarray*}
are $C^{l-4}$-smooth, $4\pi$-periodic in $\psi$ and $2\pi$-periodic
in $\beta t$. $R_{13},R_{14},$ $R_{23},R_{24},$ $R_{33},$ and
$R_{34}$ are $C^{l-4}$-smooth functions of $(h,\psi,\varphi,\beta t,\alpha t,\mu,\varepsilon)$,
$4\pi$-periodic in $\psi$ and $2\pi$-periodic in $\varphi\beta t,\alpha t.$

\section{Existence of the perturbed manifold} \label{sec:Existence}

Using the local coordinates introduced in the previous section, we
investigate here the existence and properties of the perturbed manifold.
In addition to the circle $\mathbb{T}_{1}=\mathbb{R}/(2\pi\mathbb{Z})$
we will use the notation $\mathbb{T}'_{1}=\mathbb{R}/(4\pi\mathbb{Z})$
for the circle of length $4\pi$ and $\mathbb{T}_{k}=\underbrace{\mathbb{T}_{1}\times\cdots\times\mathbb{T}_{1}}_{k\,\mathrm{times}}$
for $k-$dimensional torus.
\begin{lemma}
\label{lemma2} For $\mu\in[0,\mu_{0}]$ and $\varepsilon\in[0,\varepsilon_{0}]$
with sufficiently small $\mu_{0}$ and $\varepsilon_{0}$, the system
(\ref{ss1})--(\ref{ss3}) has an integral manifold
$$
\mathfrak{M}_{\mu,\varepsilon}=\{(h,\psi,\varphi,t):\, h=u(\psi,\varphi,\beta t,\alpha t,\mu,\varepsilon),\,(\psi,\varphi)\in\mathbb{T}'_{1}\times\mathbb{T}_{1},t\in\mathbb{R}\},
$$
 where the function $u$ has the form \begin{eqnarray}
& & u(\psi,\varphi,\beta t,\alpha t,\mu,\varepsilon)=\mu^{2}u_{0}(\psi,\beta t,\mu)+ \nonumber\\[2mm]
& & \hspace{15mm} + \varepsilon\mu^{2}u_{1}(\psi,\varphi,\beta t,\alpha t,\mu,\varepsilon)+
\varepsilon^{2}\mu u_{2}(\psi,\varphi,\beta t,\alpha t,\mu,\varepsilon)\label{mani00}\end{eqnarray}
with $C^{l-4}$-smooth $4\pi$-periodic in $\psi,$ $2\pi$-periodic
in $\varphi,\beta t,\alpha t$ functions $u_{0}, u_{1}$ and $u_{2}$ such
that $\|u_{j}\|_{C^{l-4}}\le M_{1},\ j = 0,1,2,$
where positive constant $M_{1}$ does not depend on $\alpha,\mu,\varepsilon.$
Here $\left\Vert \cdot\right\Vert _{C^{l-4}}$ is the norm of functions
from $C^{l-4}(\mathbb{T}'_{1}\times\mathbb{T}_{3})$ with fixed parameters
$\mu$ and $\varepsilon$.

The integral manifold $\mathfrak{M}_{\mu,\varepsilon}$ is asymptotically
stable in the following sense: there exists $\nu_{0}=\nu_{0}(\mu_{0},\varepsilon_{0})$
such that for every initial value $(h,\psi,\varphi)$ at time $\tau$
with $\|h\|\le\nu_{0}$, there exists a unique $(\psi_{0},\varphi_{0})$
such that
\begin{eqnarray}
 &  & \|N(t,\tau,h,\psi,\varphi)-N(t,\tau,u(\psi_{0},\varphi_{0},
 \beta\tau,\alpha\tau,\mu,\varepsilon),\psi_{0},\varphi_{0})\|\nonumber \\[2mm]
 &  & \hspace{10mm} \le Le^{-\kappa(t-\tau)}\|(h,\psi,\varphi)-(u(\psi_{0},\varphi_{0},
 \beta\tau,\alpha\tau,\mu,\varepsilon),\psi_{0},\varphi_{0})\|, \ t \ge \tau,
\nonumber
\end{eqnarray}
where constants $L\ge1$ and $\kappa>0$ do not depend on $\alpha,\mu,\varepsilon$.
$N(t,\tau,h,\psi,\varphi)$ is the solution of the system (\ref{ss1})
- (\ref{ss3}) with an initial value $N(\tau,\tau,h,\psi,\varphi)=(h,\psi,\varphi)$.\end{lemma}
\begin{proof}
Setting $\zeta_{1}=\beta t,\zeta_{2}=\alpha t$ in the system (\ref{ss1})
- (\ref{ss3}), we obtain an autonomous system \begin{eqnarray}
 &  & \frac{dh}{dt}=Hh+Q_{1}(h,\psi,\varphi,\zeta_{1},\zeta_{2},\mu,\varepsilon),\label{ss11}\\
 &  & \frac{d\psi}{dt}=\beta_{0}+Q_{2}(h,\psi,\varphi,\zeta_{1},\zeta_{2},\mu,\varepsilon),
\label{ss22}
\\
 &  & \frac{d\varphi}{dt}=\alpha_{0}+Q_{3}(h,\psi,\varphi,\zeta_{1},\zeta_{2},\mu,\varepsilon),\label{ss33}
\\
 &  & \frac{d\zeta_{1}}{dt}=\beta,\ \ \frac{d\zeta_{2}}{dt}=\alpha,\label{ss44}\end{eqnarray}
where $C^{l-4}$-smooth functions $Q_{1},Q_{2}$ and $Q_{3}$ are
obtained from the right-hand sides of (\ref{ss1})--(\ref{ss3})
with regard in $\beta t=\zeta_{1},\alpha t=\zeta_{2}.$ The corresponding
reduced system has the form \[
\frac{dh}{dt}=Hh,\quad\frac{d\psi}{dt}=\beta_{0},
\quad\frac{d\varphi}{dt}=\alpha_{0},
\quad\frac{d\zeta_{1}}{dt}=\beta,\quad\frac{d\zeta_{2}}{dt}=\alpha.\]
The eigenvalues of the constant matrix $H$ have negative real parts,
hence \begin{eqnarray}
\|e^{Ht}\|\le\mathcal{L}e^{-\kappa_{0}t},\ t>0,\label{1**}\end{eqnarray}
 where $\mathcal{L}=\mathrm{const}\ge1,\ \kappa_0=\mathrm{const}>0.$


By introducing new variables $\zeta = (\psi,\varphi,\zeta_1,\zeta_2)$ and new parameters $\lambda = (\eta_{1}$, $\eta_{2}$, $\eta_{3}, \mu,\varepsilon)$
the following system
\begin{eqnarray}
 & & \frac{dh}{dt}=Hh + \tilde Q_1(h, \zeta, \lambda), \label{ss1-extended}\\[3mm]
 & & \frac{d\zeta}{dt} = \omega_0 + \tilde Q(h, \zeta, \lambda), \label{ss2-extended}
\end{eqnarray}
coincides with (\ref{ss1})--(\ref{ss3}) if $\eta_{1}=\mu^{2}$,
$\eta_{2}=\varepsilon\mu^{2}$, $\eta_{3}=\varepsilon^{2}\mu,$
$\zeta_{1}=\beta t,\zeta_{2}=\alpha t,$ and
\begin{eqnarray*}
& & \tilde Q_1(h, \zeta, \lambda) = \eta_{1}R_{11}+ R_{12}+\eta_{2}R_{13}+\eta_{3}R_{14}, \\
& &\tilde Q = (\tilde Q_2, \tilde Q_3, \tilde Q_4, \tilde Q_5),  \ \ \omega_0 = (\beta_0, \alpha_0, \beta, \alpha), \\
& & \tilde Q_2 = \eta_{1}p^{T}(\psi)\mathcal{G}(z_{0}(\psi),\beta t)+\eta_{1}R_{21}+ R_{22}+\eta_{2}R_{23}+\eta_{3}R_{24}, \\
& & \tilde Q_3 = \eta_{1}R_{31}+ R_{32}+\eta_{2}R_{33}+\eta_{3}R_{34}, \ \ \tilde Q_4 =  \tilde Q_5 = 0.
\end{eqnarray*}

By \cite{Samoilenko1991} or  \cite{Yi1993b},
for all parameters  $\lambda\in I_{\lambda_{0}}=\left\{ \lambda:\,\|\lambda\|\le\lambda_{0}\right\} $,
with  sufficiently small $\lambda_{0}$
system (\ref{ss1-extended})--(\ref{ss2-extended}) has a unique
invariant manifold \begin{equation}
h=w_0(\zeta,\lambda), \ \ \zeta \in \mathbb{T}'_1 \times \mathbb{T}_3, \  \lambda\in I_{\lambda_{0}},   \label{eq:InvMan}\end{equation}
where $w_0(\zeta,\lambda)$ is bounded Lipschitz in $\zeta, \lambda$ and $w_0(\zeta,\lambda) \to 0$ uniformly as $(\eta_1,\eta_2,\eta_3) \to 0.$

In order to show this, for $\lambda\in I_{\lambda_{0}}$ the mapping
$T_\lambda: \, \mathcal{F}_\rho \to \mathcal{F}_\rho$
has been used,
\begin{eqnarray}
T_\lambda(w)(\zeta)= \int_{-\infty}^{0}e^{-H\tau}\tilde Q_{1}\left(w\left(
\zeta_{\tau},\lambda\right),\zeta_{\tau}, \lambda\right)d\tau,\nonumber 
\end{eqnarray}
where $\zeta_{\tau}$ is solution of (\ref{ss2-extended}) for $h = w(\zeta, \lambda)$
with initial conditions $\zeta_{0} = \zeta.$
$\mathcal{F}_\rho$ is the space of Lipschitz continuous functions
$w: \, \mathbb{T}_{1}'\times\mathbb{T}_{3} \to \mathbb{R}^n$ such that
$\| w \|_C \le \rho, \ {\rm Lip}\, w \le \rho,$  \ ${\rm Lip}\, w$ is Lipschitz constant of
$w$ with respect to $\zeta.$

Denote $\eta = \eta_1 +  \eta_2 +  \eta_3.$ Let $M_0$ be a positive constant such that
$$\|D^j R_{11}\| \le M_0, \ \|D^j R_{12}\| \le \|h\|^2 M_0, \ \|D^j R_{13}\| \le M_0, \ \|D^j R_{14}\| \le M_0,$$
for $\|h\| \le \rho_0, |\lambda| \le \lambda_0, \zeta \in \mathbb{T}'_1 \times \mathbb{T}_3$
with some $\rho_0 > 0, \lambda_0 > 0.$ $D^j$ are derivatives of order $|j| \le l-4$ with respect to $h,\zeta,\lambda$
(first derivatives of $R_{12}$ with respect to $h$ have estimate $\|h\| M_0$ and higher derivatives have 
estimate $M_0$).

We consider the subset $\mathcal{F}_{\eta a_0}$ of $\mathcal{F}_{\rho_0}$ which consists of functions $w$ with
$\|w\|_C \le \eta a_0, Lip_\zeta w \le \eta a_0,$ where $a_0$ is some positive constant.

For sufficiently small $\eta,$ the mapping
\begin{equation} \label{mapping-T}
T_{\lambda(\eta)}: \ \mathcal{F}_{\eta a_0} \to \mathcal{F}_{\eta a_0}
\end{equation}
is well defined. Here $\lambda(\eta)$ means $\lambda = (\eta_{1}$, $\eta_{2}$, $\eta_{3}, \mu,\varepsilon)$
with $\eta_{1} + \eta_{2} + \eta_{3} =\eta.$
Really, for $w \in \mathcal{F}_{\eta a_0},$ the function $\tilde Q_1$ has the following estimate
$$\|\tilde Q_1(w(\zeta,\lambda),\zeta,\lambda)\|_C \le \eta M_0 + \eta^2 a_0^2 M_0,$$
hence, taking into account (\ref{1**}),
\begin{eqnarray} \label{otsinkaT}
\|T_{\lambda(\eta)}(w)\|_C \le \frac{\mathcal{L}}{\kappa_0}(\eta + \eta^2 a_0^2)M_0.
\end{eqnarray}

Let $\zeta_{\tau}^1$ and $\zeta_{\tau}^2$ be two solutions of (\ref{ss2-extended})
with $h = w(\zeta,\lambda), \|w\|_C \le \eta a_0,$ $Lip_\zeta w \le \eta a_0$ and initial values $\zeta_{0}^1$ and $\zeta_{0}^2.$
Then
\begin{eqnarray} \label{otsinka-zeta}
\| \zeta_{t}^1 - \zeta_{t}^2 \| \le \| \zeta_{0}^1 - \zeta_{0}^2 \| e^{(\eta a_0 Lip_h \tilde Q +
Lip_\zeta \tilde Q)t} \le \| \zeta_{0}^1 - \zeta_{0}^2 \| e^{\eta a_1 t},
\end{eqnarray}
where $a_1$ is a positive constant independent on $\eta.$
Inequality (\ref{otsinka-zeta}) permits to estimate Lipschitz constant of $T(w):$
\begin{eqnarray}
& & \| T_{\lambda(\eta)}(w)(\zeta^1_0) - T_{\lambda(\eta)}(w)(\zeta_0^2)\| \le \nonumber \\
& & \le \int_{\infty}^0 \mathcal{L} e^{-\kappa_0\tau}\left( Lip_h \tilde Q_1 Lip_\zeta w+ Lip_\zeta \tilde Q_1\right)
\| \zeta_{\tau}^1 - \zeta_{\tau}^2 \| d\tau \le \nonumber \\
& & \le \frac{ \mathcal{L}}{\kappa_0 - a_1 \eta} \left( \eta a_0 Lip_h \tilde Q_1 + Lip_\zeta \tilde Q_1\right)
\| \zeta_{0}^1 - \zeta_{0}^2 \|. \label{otsinkaTL}
 \end{eqnarray}
One can verify that
\begin{equation} \label{516-1}
Lip_h \tilde Q_1 \le \eta M_0 + 3 \eta a_0 M_0, \ \ Lip_\zeta \tilde Q_1 \le \eta M_0 + \eta^2 a_0^2 M_0
\end{equation}
if $\|h\| \le \eta a_0.$

There exist positive $a_0$ and $\eta_0$ such that
$$\frac{ \mathcal{L}}{\kappa_0 - a_1 \eta} \left( \eta a_0 Lip_h \tilde Q_1 + Lip_\zeta \tilde Q_1\right) \le \eta a_0, \ \ \frac{ \mathcal{L} M_0}{\kappa_0}(\eta + \eta^2 a_0^2) \le \eta a_0$$
for all $\eta \le \eta_0.$
Taking into account (\ref{516-1}), to this end it suffices
$$\frac{a_0 \kappa_0}{M_0 \mathcal{L}} - 1 \ge \eta a_0^2, \ \
\frac{M_0 \mathcal{L}}{\kappa_0 - a_1\eta}(1 + 4\eta a_0^2 + a_0\eta) \le a_0.$$
Hence, mapping (\ref{mapping-T}) is well defined for $\eta \le \eta_0.$

Analogously to  \cite{Yi1993b} (Theorem~6.1), we show that the map $T_{\lambda(\eta)}(w)$
is a contraction of set $\mathcal{F}_{\eta a_0}$ for all $\eta \le \eta_1$ with some $\eta_1 \le \eta_0.$
The mapping $T_{\lambda(\eta)}$ has unique fixed point $w_0(\zeta, \lambda)$ for all $\lambda \in I_{\lambda_0}$
with $\eta \le \eta_1.$

Expressions in right-hand sides of (\ref{otsinkaT}) and (\ref{otsinkaTL}) don't depend on $\alpha \in
[\alpha_0,\infty)$ (note, that $\alpha$ is contained explicitly only in equation $d\zeta_2/dt = \alpha$).
Hence, values $a_0$ and $\eta_0$ can be chosen independent on $\alpha \in
[\alpha_0,\infty).$
By construction, $w_0$ satisfies
$\|w_0(\zeta, \lambda)\| \le \eta a_0$ with positive constant $a_0$ independent on $\alpha.$


Note that by \cite{Yi1993b},
for sufficiently small $\lambda$ the map
 $$T_{\lambda}(w):
C^{l-4}(\mathbb{T}_{1}'\times\mathbb{T}_{3}, \mathbb{R}^n) \to C^{l-4}(\mathbb{T}_{1}'\times\mathbb{T}_{3}, \mathbb{R}^n)$$
is well defined.

For proving $C^{l-4}$ smoothness of integral manifold $w_0(\zeta,\lambda)$ we use the fiber contraction theorem \cite{Chicone1999}, p. 127.
At first we show that invariant manifold is $C^1$ with respect to $\zeta.$ The continuous differentiability with respect to $\lambda$
is proved analogously.
The smoothness up to $C^{l-4}$ can be improved inductively.

Following \cite{Chicone1999}, p. 336, we introduce the set $\mathcal{F}^1$ of all bounded continuous functions $\Phi$ that map $\mathbb{T}_{1}'\times\mathbb{T}_{3}$
into the set of all $n\times 4$ matrices. Let $\mathcal{F}^1_{\rho}$ denote the closed ball in $\mathcal{F}^1$ with radius $\rho.$

For $w \in\mathcal{F}_{\eta a_0},$ we consider the map $T^1_\lambda(w,\Phi): \, \mathcal{F}_{\eta a_0}
\times \mathcal{F}^1_{\eta a_2}  \to \mathcal{F}^1_{\eta a_2},$
\begin{eqnarray}
& & T^1_\lambda(w,\Phi)(\zeta) = \int_{-\infty}^0 e^{-H\tau} \Biggl( \frac{\partial \tilde Q_1(w(\zeta_\tau, \lambda), \zeta_\tau, \lambda)}{\partial \zeta}
 + \nonumber \\
& & + \frac{\partial \tilde Q_1(w(\zeta_\tau, \lambda), \zeta_\tau, \lambda)}{\partial h}\Phi(\zeta_\tau,\lambda)\Biggl)W(\tau,\lambda)d\tau,
\label{oznT1}
\end{eqnarray}
 where $\zeta_t, W(t,\lambda)$ are solutions of the system
\begin{eqnarray}
& & \frac{d\zeta}{dt} = \omega_0 + \tilde Q (w(\zeta, \lambda), \zeta,\lambda),
\label{eq-1w}
\\
& & \frac{dW}{dt} = \frac{\partial \tilde Q(w(\zeta, \lambda), \zeta,\lambda)}{\partial \zeta}W +
\frac{\partial \tilde Q(w(\zeta, \lambda), \zeta,\lambda)}{\partial h}\Phi(\zeta,\lambda)W.  \label{eq-2w}
\end{eqnarray}
Taking into account the structure of the function $\tilde Q,$ we see that
$$\|\frac{\partial \tilde Q(w, \zeta,\lambda)}{\partial \zeta}\| \le \eta K, \
\|\frac{\partial \tilde Q(w, \zeta,\lambda)}{\partial h}\| \le K$$
with some positive constant $K$ independent on $\eta.$
Choosing $\eta$ such that  $K\eta(1 + a_2) \le \kappa_0/4$ and applying Gronwall's
inequality, we obtain
\begin{eqnarray} \label{W1}
\|W(t,\lambda)\| \le M e^{(\kappa_0/4)(t-t_0)},
\end{eqnarray}
where $M$ is some positive constant.

Taking into account (\ref{W1}) and inequalities
$$\left\| \frac{\partial \tilde Q_1}{\partial \zeta}\right\| \le M_0 \eta + M_0 \eta^2 a_0^2, \
\left\| \frac{\partial \tilde Q_1}{\partial h}\right\| \le M_0 \eta + 3M_0 \eta a_0,$$
we get
\begin{eqnarray*} 
& &  \| T^1_\lambda(w,\Phi)\| \le \int_{-\infty}^0 \mathcal{L} e^{-\kappa_0 \tau}
(1 + \eta a_0^2 + \eta a_2 + 3\eta a_0 a_2)\eta M_0 M e^{\kappa_0 \tau/4} d\tau \le \nonumber\\
& & \le \frac{4\mathcal{L} M M_0}{3\kappa_0}(1 + \eta a_0^2 + \eta a_2 + 3\eta a_0 a_2)\eta.
\end{eqnarray*}
There exist $a_2 > 0$ and $\eta_2$ such that the last expression is less then $\eta a_2$ for $\eta \le \eta_2.$
Hence,the mapping $T^1_\lambda(w,\Phi)$ is well defined.

Let us consider the mapping
\begin{eqnarray} \label{TT1}
(w,\Phi) \to (T_{\eta a_0}(w), T^1_{\lambda}(w,\Phi)).
\end{eqnarray}
Analogously to \cite{Chicone1999}, p. 337, it can be shown that (\ref{TT1}) is continuous with respect to $w.$
Now we prove that the mapping (\ref{TT1})
 is a fiber contraction.
For $w \in \mathcal{\mathcal{F}}_{\eta a_0}$ and $\Phi_1, \Phi_2 \in \mathcal{\mathcal{F}}_{\eta a_2}^1$
we get
\begin{eqnarray}
& & \|  T^1_\lambda(w,\Phi_1) - T^1_\lambda(w,\Phi_2) \| \le \nonumber \\
& & \le \|\int_{-\infty}^0 e^{-H\tau} \Biggl( \frac{\partial \tilde Q_1}{\partial \zeta}(W_1 - W_2)  + \frac{\partial \tilde Q_1}{\partial h}(\Phi_1 W_1 - \Phi_2 W_2)\Biggl)d\tau\| \le \nonumber \\
& & \le \int_{-\infty}^0 e^{-H\tau}\Biggl(\|\frac{\partial \tilde Q_1}{\partial h}\| \|W_2\| \|\Phi_1 - \Phi_2\|
+ \nonumber \\
& & + \Bigl( \|\frac{\partial \tilde Q_1}{\partial \zeta}\| + \|\frac{\partial \tilde Q_1}{\partial h}\| \|\Phi_1\|\Bigl)\|W_1 - W_2\|\Biggl)d\tau. \label{T1}
\end{eqnarray}
By (\ref{eq-2w}), we obtain following estimate for $\|W_1 - W_2\|:$
\begin{eqnarray*}
& & \frac{d(W_1 - W_2)}{dt} = \frac{\partial \tilde Q(w, \zeta,\lambda)}{\partial \zeta}(W_1 - W_2) + \frac{\partial \tilde Q(w, \zeta,\lambda)}{\partial h}(\Phi_1 W_1 - \Phi_2 W_2), \\
& & \|W_1(t,\lambda) - W_2(t,\lambda)\| \le \int_0^t \left( \|\frac{\partial \tilde Q}{\partial \zeta}\| +
\|\frac{\partial \tilde Q}{\partial h}\| \|\Phi_1\|\right)\|W_1(s,\lambda) - W_2(s,\lambda)\|ds + \\
& & + \int_0^t \|\frac{\partial \tilde Q}{\partial h}\| \| W_2\| \| \Phi_1 - \Phi_2\|_C ds.
\end{eqnarray*}
Inserting (\ref{W1}) into the second integral and applying the Gronwall's inequality, we get
\begin{eqnarray} \label{W2}
 \|W_1(t,\lambda) - W_2(t,\lambda)\| \le \frac{4 K M}{\kappa_0} e^{(\kappa_0/2)t} \| \Phi_1 - \Phi_2\|_C.
 \end{eqnarray}

Putting (\ref{W2}) into (\ref{T1}), we obtain
$$\|  T^1_{\lambda}(w,\Phi_1) - T^1_{\lambda}(w,\Phi_2) \| \le \varsigma \| \Phi_1 - \Phi_2\|_C,$$
where
$$\varsigma = \frac{2\mathcal{L}M M_0\eta}{\kappa_0}\left( 1 +  a_0 +
\frac{4K}{\kappa_0}(1 + \eta a_0^2 + \eta a_2 + 3\eta a_0 a_2)\right).$$
We can choose $\varsigma < 1$ for sufficiently small $\eta$ hence
the mapping (\ref{TT1}) is a fiber contraction. It has unique globally attracting fixed point $(w_0,w_1).$
By (\ref{oznT1}), it is easy to see that $w_1(\zeta,\lambda)$ is bounded uniformly to $\alpha \in [\alpha_0, \infty).$
Repeating \cite{Chicone1999}, p.296, one can show that $w_0$ is continuously differentiable and $Dw_0 = w_1.$
\bigskip

Taking into account that the invariant manifold (\ref{eq:InvMan})
for $\eta_{1}=\eta_{2}=\eta_{3}=0$ equals to zero $h=0$, it can
be represented as \[
h=\eta_{1}w_{0}(\psi,\zeta_{1},\mu,\eta_{1})+\eta_{2}w_{1}(\psi,\varphi,\zeta_{1},\zeta_{2},\lambda)+
\eta_{3}w_{2}(\psi,\varphi,\zeta_{1},\zeta_{2},\lambda).\]
Note that $w_{0}$ does not depend on $\zeta_{2}$, $\eta_{2},\eta_{3}$,
and $\varepsilon$, since system (\ref{ss1-extended})--(\ref{ss2-extended})
is independent on $\zeta_{2}$ for $\eta_{2}=\eta_{3}=\varepsilon=0$.
Taking into account the dependence of $\eta_{1},$ $\eta_{2}$, and
$\eta_{3}$ on $\mu$ and $\varepsilon$, we obtain that the invariant
manifold of (\ref{ss11})--(\ref{ss44}) has the following form
\begin{eqnarray}
& & h=u(\psi,\varphi,\zeta_{1},\zeta_{2},\mu,\varepsilon)=\mu^{2}u_{0}(\psi,\zeta_{1},\mu)+ \nonumber\\[2mm]
& & \hspace{15mm} +\varepsilon\mu^{2}u_{1}(\psi,\varphi,\zeta_{1},\zeta_{2},\mu,\varepsilon)+
\varepsilon^{2}\mu u_{2}(\psi,\varphi,\zeta_{1},\zeta_{2},\mu,\varepsilon). \label{5.144}
\end{eqnarray}
Respectively, system (\ref{ss1}) - (\ref{ss3}) has integral manifold $\mathfrak{M}_{\mu,\varepsilon}$
defined by the function $u(\psi,\varphi,\beta t,\alpha t,\mu,\varepsilon).$

Since manifold (\ref{5.144}) is smooth, it satisfies the following relation
\begin{eqnarray}
 &  & \frac{\partial u}{\partial\psi}(\beta_{0}+Q_{2}(u,\psi,\varphi,\zeta_{1},\zeta_{2},\mu,\varepsilon))+\frac{\partial u}{\partial\varphi}(\alpha_{0}+Q_{3}(u,\psi,\varphi,\zeta_{1},\zeta_{2},\mu,\varepsilon))\nonumber \\
 &  & +\frac{\partial u}{\partial\zeta_{1}}\beta+\frac{\partial u}{\partial\zeta_{2}}\alpha=Hu+Q_{1}(u,\psi,\varphi,\zeta_{1},\zeta_{2},\mu,\varepsilon).\end{eqnarray}
 Taking into account this expression and performing the change of
variables $h=\tilde{h}+u(\psi,\varphi,\zeta_{1},\zeta_{2},\mu,\varepsilon)$
in system (\ref{ss11})--(\ref{ss44}), we obtain \begin{eqnarray}
\frac{d\tilde{h}}{dt}=\left(H+Q_0(\tilde{h},\psi,\varphi,\zeta_{1},\zeta_{2},\mu,\varepsilon)\right)\tilde{h},\label{lin1}\end{eqnarray}
 where \begin{eqnarray*}
 &  & Q_0(\tilde{h},\psi,\varphi,\zeta_{1},\zeta_{2},\mu,\varepsilon)\tilde{h}=Q_{1}(u+\tilde{h},\psi,\varphi,\zeta_{1},\zeta_{2},\mu,\varepsilon)-Q_{1}(u,\psi,\varphi,\zeta_{1},\zeta_{2},\mu,\varepsilon)\\
 &  & -\frac{\partial u}{\partial\psi}(Q_{2}(u+\tilde{h},\psi,\varphi,\zeta_{1},\zeta_{2},\mu,\varepsilon)-Q_{2}(u,\psi,\varphi,\zeta_{1},\zeta_{2},\mu,\varepsilon))\\
 &  & -\frac{\partial u}{\partial\varphi}(Q_{3}(u+\tilde{h},\psi,\varphi,\zeta_{1},\zeta_{2},\mu,\varepsilon)-Q_{3}(u,\psi,\varphi,\zeta_{1},\zeta_{2},\mu,\varepsilon)).\end{eqnarray*}
The function $Q_0$ can be represented as a sum of two terms
$Q_0=Q_{01}+Q_{02}$ such that $Q_{01}=\mathcal{O}(\|\tilde{h}\|)$
and $Q_{02}=\mathcal{O}(\mu)$. Therefore there exist $a_{0}>0$
and $\mu_{1}>0$ such that $\|Q_0\|_{C}<\kappa_{0}/\left(2\mathcal{L}\right)$
for all $\tilde{h}$ and $\mu$ with $\|\tilde{h}\|\le a_{0}$ and
$\mu\le\mu_{1}$. Here $a_{0}$ does not depend on $\mu_{1}$. 
Taking into account (\ref{1**}) and an estimate of the fundamental solution for 
perturbed linear system  \cite{Samoilenko1991}, we obtain the
following estimate for solutions of (\ref{lin1}): \begin{eqnarray}
\|\tilde{h}(t)\|\le\|\tilde{h}(t_{0})\|\mathcal{L}e^{(-\kappa_{0}+\mathcal{L}\|Q_0\|_{C})(t-t_{0})}\le\|\tilde{h}(t_{0})\|\mathcal{L}e^{-(\kappa_{0}/2)(t-t_{0})}.\label{lin2}\end{eqnarray}
Since $u$ is proportional to $\mu$ and $a_{0}$ does not depend
on $\mu$, for all small enough $\mu_{1}$ it holds $h_{0}=a_{0}-\sup_{0\le\mu\le\mu_{1}}\|u\|_{C}>0$.
Taking into account that $h=\tilde{h}+u$, one can conclude that for
all $h$ with $\|h\|\le h_{0}$ the inequality $\|\tilde{h}\|\le a_{0}$
and estimate (\ref{lin2}) hold.

As result, if $\mu\le\mu_{1}$ and solution $(h(t),\psi(t),\varphi(t))$
of (\ref{ss11})--(\ref{ss44}) satisfies the condition $\|h(t_{0})\|\le h_{0}$
at initial moment of time $t=t_{0}$ then
\begin{eqnarray}
 & & \|h(t)-u(\psi(t),\varphi(t),\beta t,\alpha t,\mu,\varepsilon)\|\nonumber \\[2mm]
 & & \hspace{15mm} \le\mathcal{L}e^{-\frac{\kappa_{0}}{2}(t-t_{0})}\|h(t_{0})-
 u(\psi(t_{0}),\varphi(t_{0}),\beta t_{0},\alpha t_{0},\mu,\varepsilon)\|\label{eq:514}\end{eqnarray}
for all $t\ge t_{0}.$

By \cite{Samoilenko1991} and \cite{Yi1993a}, the integral manifold $\mathfrak{M}_{\mu,\varepsilon}$
is asymptotically stable, i.e. there exists $\nu_{1}=\nu_{1}(\mu_{0},\varepsilon_{0})$
such that if $\rho((h,\psi,\varphi),\mathfrak{M}_{\mu,\varepsilon}(\tau))\le\nu_{1}$
at time $\tau$ then there is a unique $(\psi_{0},\varphi_{0})$ such
that \begin{eqnarray}
 & & \|N(t,\tau,h,\psi,\varphi)-N(t,\tau,u(\psi_{0},\varphi_{0},\beta\tau,\alpha\tau,\mu,\varepsilon),\psi_{0},\varphi_{0})\|\nonumber \\[2mm]
 & &  \hspace{6mm} \le Le^{-\kappa(t-\tau)}\|(h,\psi,\varphi)-(u(\psi_{0},\varphi_{0},\beta\tau,
 \alpha\tau,\mu,\varepsilon),\psi_{0},\varphi_{0})\|, \ t\ge\tau, \label{mani2-1}\end{eqnarray}
where constants $L\ge 1$  $\kappa > 0$ don't
depend on $\alpha,\mu,\varepsilon,$ $\rho(.,.)$ is the metric in
$\mathbb{R}^{n}\times\mathbb{T}'_{1}\times\mathbb{T}_{1},$ $N(t,\tau,h,\psi,\varphi)$
is the solution of the system (\ref{ss1})--(\ref{ss3})
with an initial value $N(\tau,\tau,h,\psi,\varphi)=(h,\psi,\varphi)$,
and $\mathfrak{M}_{\mu,\varepsilon}(\tau)$ is the cross-section of
$\mathfrak{M}_{\mu,\varepsilon}$ for $t=\tau:$ \[
\mathfrak{M}_{\mu,\varepsilon}(\tau)=\{(u(\psi,\varphi,\beta\tau,\alpha\tau,\mu,\varepsilon),\psi,\varphi):\,(\psi,\varphi)\in\mathbb{T}'_{1}\times\mathbb{T}_{1}\}.\]

Inequalities (\ref{eq:514}) and (\ref{mani2-1}) assure the exponential
attraction of all solutions of (\ref{ss1})--(\ref{ss3})
that start at $t=t_{0}$ from a small neighborhood of the unperturbed
manifold $h=0$ to solutions on the perturbed manifold $\mathfrak{M}_{\mu,\varepsilon}$
with the rate of attraction, which is independent on $\mu\in(0,\mu_{0}],\varepsilon\in(0,\varepsilon_{0}],\alpha\ge\alpha_{*}$. \end{proof}
\begin{corollary}
The system (\ref{f1})--(\ref{f2}) has the integral manifold
\begin{eqnarray}
\mathfrak{M}_{\mu,\varepsilon}^{0} & = & \{(z_{0}(\psi)+\mu^{2}\Phi(\psi)u_{0}(\psi,\beta t,\mu)+\varepsilon\mu^{2}\Phi(\psi)u_{1}(\psi,\varphi,\beta t,\alpha t,\mu,\varepsilon)
\nonumber
\\
 &  & +\varepsilon^{2}\mu\Phi(\psi)u_{2}(\psi,\varphi,\beta t,\alpha t,\mu,\varepsilon),\psi,\varphi,t):\ (\psi,\varphi)\in\mathbb{T}_{1}\times\mathbb{T}'_{1},t\in\mathbb{R}\}.\nonumber
\end{eqnarray}

\end{corollary}

\section{Investigation of the system on the manifold \label{sec:Investigation-of-the}}

Substituting the expression for the invariant manifold (\ref{mani00})
into the equations (\ref{ss2})--(\ref{ss3}), we obtain the system
on the manifold \begin{eqnarray}
 &  & \frac{d\psi}{dt}=\beta_{0}+\mu^{2}p^{T}(\psi)\mathcal{G}(z_{0}(\psi),\beta t)+\mu^{4}S_{11}(\psi,\beta t,\mu)
\nonumber \\
 & &  \hspace{15mm}+ \varepsilon\mu^{2}S_{12}(\psi,\varphi,\beta t,\alpha t,\mu,\varepsilon))
 +\varepsilon^{2}\mu S_{13}(\psi,\varphi,\beta t,\alpha t,\mu,\varepsilon),  \label{m1} \\
 &  & \frac{d\varphi}{dt}=\alpha_{0}+\mu^{2}S_{21}(\psi,\beta t,\mu)+\varepsilon\mu^{2}S_{22}(\psi,\varphi,\beta t,\alpha t,\mu,\varepsilon)+ \nonumber \\
& & \hspace{15mm}+ \varepsilon^{2}\mu S_{23}(\psi,\varphi,\beta t,\alpha t,\mu,\varepsilon),\label{m2}\end{eqnarray}
where $C^{l-4}$-smooth functions $S_{j},j=1,2,3$ are periodic in
$\psi,\varphi,\beta t,\alpha t.$

Now we assume that the frequencies $\beta_{0}$ and $\beta$ are close
to each other \[
\beta-\beta_{0}=\mu^{2}\Delta.\]
In the system (\ref{m1})--(\ref{m2}), we change the variables according
to the formula \[
\psi=\beta t+\psi_{1}\]
and obtain the following system \begin{eqnarray}
& & \frac{d\psi_{1}}{dt} = -\mu^{2}\Delta+\mu^{2}p^{T}(\beta t+\psi_{1})\mathcal{G}(z_{0}(\beta t+\psi_{1}),\beta t)+\mu^{4}S_{11}(\beta t+\psi_{1},\beta t,\mu) \nonumber \\
 &  & \quad \quad + \varepsilon\mu^{2}S_{12}(\beta t+\psi_{1},\varphi,\beta t,\alpha t,\mu,\varepsilon)+\varepsilon^{2}\mu S_{13}(\beta t+\psi_{1},\varphi,\beta t,\alpha t,\mu,\varepsilon), \label{eq:theta-1} \\
 & &  \frac{d\varphi}{dt} =  \alpha_{0}+\mu^{2}S_{21}(\beta t+\psi_{1},\beta t,\mu)+\varepsilon\mu^{2}S_{22}(\beta t+\psi_{1},\varphi,\beta t,\alpha t,\mu,\varepsilon)\nonumber\\
 &  & \hspace{45mm}+\varepsilon^{2}\mu S_{23}(\beta t+\psi_{1},\varphi,\beta t,\alpha t,\mu,\varepsilon). \label{eq:phi-1} \end{eqnarray}

Performing now the change of variables \begin{eqnarray*}
 &  & \psi_{1}=\psi_{2}+\frac{\mu^{2}}{\beta}\int_{0}^{\beta t}[p^{T}(\xi+\psi_{1})\mathcal{G}(z_{0}(\xi+\psi_{1}),\xi)- G(\psi_{1})]d\xi,\\[2mm]
 &  & \varphi=\varphi_{2}+\frac{\mu^{2}}{\beta}\int_{0}^{\beta t}[S_{21}(\xi+\psi_{1},\xi,\mu)-{S}_{21}(\psi_{1},\mu)]d\xi,\\[2mm]\end{eqnarray*}
 where \[
G(\psi_{1}):=\frac{1}{2\pi}\int_{0}^{2\pi}p^{T}(\xi+\psi_{1})\mathcal{G}(z_{0}(\xi+\psi_{1}),\xi)d\xi,\]
\[
\bar{S}_{21}(\psi_{1},\mu):=\frac{1}{2\pi}\int_{0}^{2\pi}S_{21}(\xi+\psi_{1},\xi,\mu)d\xi,\]
the system (\ref{eq:theta-1})--(\ref{eq:phi-1}) takes the form
\begin{eqnarray}
 &  & \frac{d\psi_{2}}{dt}=-\Delta\mu^{2}+\mu^{2} {G}(\psi_{2})+\mu^{4}\tilde{S}_{11}(\psi_{2},\beta t,\mu)
 \nonumber\\
 &  & \hspace{15mm} + \varepsilon\mu^{2}\tilde{S}_{12}(\psi_{2},\varphi_{2},\beta t,\alpha t,\mu,\varepsilon)+
 \varepsilon^{2}\mu\tilde{S}_{13}(\psi_{2},\varphi_{2},\beta t,\alpha t,\mu,\varepsilon), \label{mm1} \\
 &  & \frac{d\varphi_{2}}{dt}=\alpha_{0}+\mu^{2}\bar{S}_{21}(\psi_{2},\mu)
 +\mu^{4}\tilde{S}_{21}(\psi_{2},\beta t,\mu) 
 \nonumber \\
 & & \hspace{15mm} + \varepsilon\mu^{2}\tilde{S}_{22}(\psi_{2},\varphi_{2},\beta t,\alpha t,\mu,\varepsilon)+
 \varepsilon^{2}\mu\tilde{S}_{23}(\psi_{2},\varphi_{2},\beta t,\alpha t,\mu,\varepsilon),\label{mm2}\end{eqnarray}
where the functions in the right hand side are $C^{l-4}$-smooth and
periodic in $\theta_{1},\varphi_{1},\beta t,\alpha t.$

Together with (\ref{mm1})--(\ref{mm2}) we consider the averaged system
\begin{eqnarray}
 &  & \frac{d\psi_{2}}{dt}=-\Delta\mu^{2}+\mu^{2}{G}(\psi_{2}),\label{mm01}\\
 &  & \frac{d\varphi_{2}}{dt}=\alpha_{0}+\mu^{2}\bar{S}_{21}(\psi_{2},\mu).\label{mm02}\end{eqnarray}
Denote \[
G_-:=\min_{\xi\in[0,2\pi]}{G}(\xi),\quad G_+:=\max_{\xi\in[0,2\pi]}{G}(\xi).\]
 Then for $\Delta=(\beta-\beta_{0})/\mu^{2}\in[{G}_{-},{G}_{+}]$
the equation
$$
\Delta={G}(\xi)
$$
 has real solutions.

Assume that $\Delta$ is a regular value of the map ${G}$, i.e.
all pre-images $\xi=\vartheta_{j}^{0}$ of $\Delta$ by ${G}(\xi)$
are non-degenerate ${G}'(\vartheta_{j}^{0})\neq0.$ Then the number
of pre-images is finite and even due to the periodicity of ${G}(\xi)$.
The signs of every two sequential values ${G}'(\vartheta_{j}^{0})$
and ${G}'(\vartheta_{j+1}^{0})$ are opposite
\[
{G}'(\vartheta_{2k-1}^{0})=\alpha_{k}>0,
\quad {G}'(\vartheta_{2k}^{0})=-\beta_{k}<0,\quad k=1,...,N.
\]

At every interval $(\vartheta_{2k-1}^{0},\vartheta_{2k}^{0})$ the
function ${G}(\theta)-\Delta$ is positive and \[
\min_{\theta\in[\vartheta_{2k-1}^{0}+\delta,\vartheta_{2k}^{0}-\delta]}
{G}(\theta)>\Delta\]
 for every sufficiently small $\delta.$

Analogously, at every interval $(\vartheta_{2k}^{0},\vartheta_{2k+1}^{0})$
the function ${G}(\theta)-\Delta$ is negative and \[
\max_{\theta\in[\vartheta_{2k}^{0}+\delta,\vartheta_{2k+1}^{0}-\delta]}{G}(\theta)
<\Delta\]
for every sufficiently small $\delta.$ Due to the periodicity of
${G}(\theta)$ we identify $\vartheta_{2N+1}^{0}$ with $\vartheta_{1}^{0}$
and $\vartheta_{0}^{0}$ with $\vartheta_{2N}^{0}.$

The averaged system (\ref{mm01})--(\ref{mm02}) has $2N$ one-dimensional
invariant manifolds \begin{eqnarray*}
\Pi_{j}^{0}=\{(\vartheta_{j}^{0},\varphi_{2}):\ \varphi_{2}\in\mathbb{T}_{1}\}.\end{eqnarray*}
 The system on the manifold $\Pi_{j}^{0}$ reduces to \begin{eqnarray*}
\frac{d\varphi_{2}}{dt}=\alpha_{0}+\mu^{2}\bar{S}_{21}(\vartheta_{j}^{0},\mu).\end{eqnarray*}
Manifolds $\Pi_{2k}^{0},k=1,...,N,$ are exponentially stable and
manifolds $\Pi_{2k-1}^{0},k=1,...,N,$ are exponentially unstable.
\begin{lemma}
\label{lemma4} There exist $\mu_{0}>0$ and $c_{0}>0$ such that
for all $0<\mu\le\mu_{0}$ and $\varepsilon\le c_{0}\sqrt{\mu}$ the system
(\ref{mm1})--(\ref{mm2}) has $2N$ integral manifolds
\begin{eqnarray}
\Pi_{j}=\left\{ \left(\psi_{2},\varphi_{2},t\right):\ \psi_{2}=\vartheta_{j}^{0}+v_{j}(\varphi_{2},\beta t,\alpha t,\mu,\varepsilon),\,\varphi_{2}\in\mathbb{T}_{1},t\in\mathbb{R}\right\} ,
\nonumber
\end{eqnarray}
where\[
v_{j}=\mu^{2}v_{j0}(\beta t,\mu)+\varepsilon v_{j1}(\varphi_{2},\beta t,\alpha t,\mu,\varepsilon)+\frac{\varepsilon^{2}}{\mu}v_{j2}(\varphi_{2},\beta t,\alpha t,\mu,\varepsilon),\]
with $C^{l-4}$ smooth, periodic in $\varphi_{2},\beta t,\alpha t$
functions $v_{jk}$, such that $\left\Vert v_{jk}\right\Vert _{C^{l-4}}\le M_{3}$
with the constant $M_{3}$ independent on $\alpha,\mu,\varepsilon.$

The manifolds $\Pi_{2k}$, $k=1,...,N,$ are exponentially stable
in the following sense: there exists $\delta_{0}$ such that if $|\psi_{20}-\vartheta_{2k}^{0}|\le\delta_{0}$
and $\varphi_{0}\in\mathbb{T}_{1}$, then there exists an unique $\varphi_{01}$
such that for $t\ge t_{0}$ the following inequality holds \begin{eqnarray}
 &  & |\psi_{2}(t,t_{0},\psi_{20},\varphi_{0})-\psi_{2}(t,t_{0},\vartheta_{2k}^{0}+v_{2k}(\varphi_{01},\beta t_{0},\alpha t_{0},\mu,\varepsilon),\varphi_{01})|\nonumber \\[2mm]
 &  & +|\varphi_{2}(t,t_{0},\psi_{20},\varphi_{0})-\varphi_{2}(t,t_{0},\vartheta_{2k}^{0}+v_{2k}(\varphi_{01},\beta t_{0},\alpha t_{0},\mu,\varepsilon),\varphi_{01})|\nonumber \\[2mm]
 &  & \le\mathcal{L}_{2}e^{-\mu^{2}\kappa_{2}(t-t_{0})}\left(|\varphi_{0}-\varphi_{01}|+
 |\psi_{20}-\vartheta_{2k}^{0}-v_{2k}(\varphi_{01},\beta t_{0},\alpha t_{0},\mu,\varepsilon)|\right),\label{sta}\end{eqnarray}
 where constants $\mathcal{L}_{2}\ge1$ and $\kappa_{2}>0$ are independent on
$\alpha,\mu$, and $\varepsilon.$

The manifolds $\Pi_{2k-1}$, $k=1,...,N,$ are exponentially unstable
in the following sense: there exists $\delta_{0}$ such that if $|\psi_{20}-\vartheta_{2k-1}^{0}|\le\delta_{0}$
and $\varphi_{0}\in\mathbb{T}_{1}$, then there exists a unique $\varphi_{01}$
such that for $t\le t_{0}$ the following inequality holds \begin{eqnarray}
 &  & |\psi_{2}(t,t_{0},\psi_{20},\varphi_{0})-\psi_{2}(t,t_{
 0},\vartheta_{2k-1}^{0}+v_{2k-1}(\varphi_{01},\beta t_{0},\alpha t_{0},\mu,\varepsilon),\varphi_{01})|\nonumber \\[2mm]
 &  & +|\varphi_{2}(t,t_{0},\psi_{20},\varphi_{0})-\varphi_{2}(t,t_{0},\vartheta_{2k-1}^{0}+v_{2k-1}(\varphi_{01},\beta t_{0},\alpha t_{0},\mu,\varepsilon),\varphi_{01})|\nonumber \\[2mm]
 &  & \le\mathcal{L}_{3}e^{\mu^{2}\kappa_{3}(t-t_{0})}\left(|\varphi_{0}-\varphi_{01}|+
 |\psi_{20}-\vartheta_{2k-1}^{0}-v_{2k-1}(\varphi_{01},\beta t_{0},
 \alpha t_{0},\mu,\varepsilon)|\right),\label{sta2-}\end{eqnarray}
 where constants $\mathcal{L}_{3}\ge1$ and $\kappa_{3}>0$ are independent on
$\alpha,\mu$, and $\varepsilon.$
\end{lemma}
\textbf{Proof.} Setting $\zeta_{1}=\beta t,\zeta_{2}=\alpha t$ in
(\ref{mm1})--(\ref{mm2}) we obtain the following autonomous system
on $4$-dimensional torus $\mathbb{T}_{4}:$
\begin{eqnarray}
 & & \frac{d\psi_{2}}{dt}=-\Delta\mu^{2}+\mu^{2}{G}(\psi_{2})+
 \mu^{4}\tilde{S}_{11}(\psi_{2},\zeta_{1},\mu) \nonumber\\
 & & \hspace{10mm} + \varepsilon\mu^{2}\tilde{S}_{12}(\psi_{2},\varphi_{2},\zeta_{1},\zeta_{2},\mu,\varepsilon)+
 \varepsilon^{2}\mu\tilde{S}_{13}(\psi_{2},\varphi_{2},\zeta_{1},\zeta_{2},\mu,\varepsilon), \label{4mm1} \\
 & & \frac{d\varphi_{2}}{dt}=\alpha_{0}+\mu^{2}\bar{S}_{21}(\psi_{2},\mu) +
 \mu^{4}\tilde{S}_{21}(\psi_{2},\zeta_{1},\mu) \nonumber\\
 & & \hspace{10mm} + \varepsilon\mu^{2}\tilde{S}_{22}(\psi_{2},\varphi_{2},\zeta_{1},\zeta_{2},\mu,\varepsilon)+
 \varepsilon^{2}\mu\tilde{S}_{23}(\psi_{2},\varphi_{2},\zeta_{1},\zeta_{2},\mu,\varepsilon),
\label{4mm2} \\
 & & \frac{d\zeta_{1}}{dt}=\beta,\qquad\frac{d\zeta_{2}}{dt}=\alpha.\label{4mm3}
 \end{eqnarray}

Let us consider a neighborhood of the point $\psi_{2}=\vartheta_{2k}^{0}$ where $k \in \{1,...,N\}.$
Neighborhoods of points $\psi_{2}=\vartheta_{2k-1}^{0}, k = 1,...,N,$ are considered analogously.
In system (\ref{4mm1})--(\ref{4mm3}), we change the variables
$\psi_{2}=\vartheta_{2k}^{0}+b_{1}$ and introduce the new time $\tau=\mu^{2}t$
\begin{eqnarray}
 &  & \frac{db_{1}}{d\tau}=-\beta_{k}b_{1}+\bar{G}_{2}(b_{1})b_{1}^{2}+\mu^{2}\tilde{S}_{11}(\vartheta_{2k}^{0}+
 b_{1},\zeta_{1},\mu) \nonumber\\
 &  & \hspace{5mm} + \varepsilon\tilde{S}_{12}(\vartheta_{2k}^{0}+b_{1},\varphi_{2},\zeta_{1},\zeta_{2},\mu,\varepsilon)+
 \frac{\varepsilon^{2}}{\mu}\tilde{S}_{13}(\vartheta_{2k}^{0}+
 b_{1},\varphi_{2},\zeta_{1},\zeta_{2},\mu,\varepsilon), \label{mm1j} \\
 &  & \frac{d\varphi_{2}}{d\tau}=\frac{\alpha_{0}}{\mu^{2}}+\bar{S}_{21}(\vartheta_{2k}^{0}+b_{1},\mu)+
 \mu^{2}\tilde{S}_{21}(\vartheta_{2k}^{0}+ b_{1},\zeta_{1},\mu) 
 \nonumber\\
 &  & \hspace{5mm}
+ \varepsilon\tilde{S}_{22}(\vartheta_{2k}^{0}+b_{1},\varphi_{2},\zeta_{1},\zeta_{2},\mu,\varepsilon)  + \frac{\varepsilon^{2}}{\mu}\tilde{S}_{23}(\vartheta_{2k}^{0}+
 b_{1},\varphi_{2},\zeta_{1},\zeta_{2},\mu,\varepsilon), \label{mm2j}
\\
 &  & \frac{d\zeta_{1}}{d\tau}=\frac{\beta}{\mu^{2}},\qquad\frac{d\zeta_{2}}{d\tau}=\frac{\alpha}{\mu^{2}},\label{mm3j}
 \end{eqnarray}
 where $\bar{G}_{2}(b_{1})b_{1}^{2}:=({G}(\vartheta_{2l}^{0}+b_{1})-\Delta)+\beta_{k}b_{1}$.

Extending the system (\ref{mm1j})--(\ref{mm3j}) by introducing new
parameters $\eta_{1}$, $\eta_{2}$, $\eta_{3}$ and $\chi$ we obtain the system
\begin{eqnarray}
 &  & \frac{db_{1}}{d\tau}=-\beta_{k}b_{1}+\bar{G}_{2}(b_{1})b_{1}^{2}+
 \eta_{1}\tilde{S}_{11}(\vartheta_{2k}^{0}+b_{1},\zeta_{1},\mu) \nonumber\\[2mm]
 &  & \hspace{5mm}+\eta_{2}\tilde{S}_{12}(\vartheta_{2k}^{0}+b_{1},\varphi_{2},\zeta_{1},\zeta_{2},\mu,\varepsilon)+
 \eta_{3}\tilde{S}_{13}(\vartheta_{2k}^{0}+b_{1},\varphi_{2},\zeta_{1},\zeta_{2},\mu,\varepsilon),\label{mm1jext} \\[2mm]
 &  & \frac{d\varphi_{2}}{d\tau}= \chi\alpha_{0} + \bar{S}_{21}(\vartheta_{2k}^{0}+b_{1},\mu)+
 \eta_{1}\tilde{S}_{21}(\vartheta_{2k}^{0}+b_{1},\zeta_{1},\mu) \nonumber \\[2mm]
 &  & \hspace{5mm}
 + \eta_{2}\tilde{S}_{22}(\vartheta_{2k}^{0}+b_{1},\varphi_{2},\zeta_{1},\zeta_{2},\mu,\varepsilon)
+\eta_{3}\tilde{S}_{23}(\vartheta_{2k}^{0}+b_{1},\varphi_{2},\zeta_{1},\zeta_{2},
 \mu,\varepsilon), \label{mm2jext}\\[2mm]
 &  & \frac{d\zeta_{1}}{d\tau}= \chi\beta, \qquad \frac{d\zeta_{2}}{d\tau} = \chi\alpha,\label{mm3jext}
 \end{eqnarray}
which coincides with (\ref{mm1j})--(\ref{mm3j}) for $\eta_{1}=\mu^{2}$,
$\eta_{2}=\varepsilon$, $\eta_{3}=\varepsilon^{2}/\mu,$ $\chi = 1 / \mu^2.$
We assume that $\lambda = (\eta_{1},\eta_{2},\eta_{1}, \mu,\varepsilon) \in I_{\lambda_0} = \{\lambda: \ \|\lambda\| \le \lambda_0\},
$ and $\chi \ge \chi_0$ with some positive $\lambda_0$ and $\chi_0.$

Let $\beta_k \in [\beta_m, \beta_M]$ with some constants $\beta_M \ge \beta_m > 0.$

We consider the function space
\begin{eqnarray} \label{space}
C^{l-4}(\mathbb{T}_3 \times I_{\lambda_0} \times [\chi_0,\infty)  \times [\beta_m, \beta_M])
 \end{eqnarray}
of bounded together with their $l-4$ derivatives functions $w(\varphi_2, \zeta_1,\zeta_2,\lambda, \chi,\beta_k)$ defined
on $(\varphi_2, \zeta_1,\zeta_2) \in \mathbb{T}_3, \ \lambda \in I_{\lambda_0}, \ \chi \in [\chi_0,\infty), \ \beta_k \in [\beta_m, \beta_M],$
and mapping
 \begin{eqnarray*}
T(w) = \int_{-\infty}^0
e^{\beta_k \xi} Q_4(w(\varphi_{2\xi},\zeta_{1\xi}, \zeta_{2\xi}, \lambda,\chi,\beta_k),\varphi_{2\xi},\zeta_{1\xi}, \zeta_{1\xi},\lambda)d\xi,
 \end{eqnarray*}
where $Q_4$ is the right hand side of (\ref{mm1jext}), and
$\varphi_{2\xi}=\varphi_2(\xi,\varphi,\zeta_{1},\zeta_{2},\lambda)$,
$\zeta_{1\xi}=\beta\xi+\zeta_{1}$, $\zeta_{2\xi}=\alpha\xi+\zeta_{2}$
is the solution of (\ref{mm2jext})--(\ref{mm3jext})
for $b_1=w(\varphi_{2},\zeta_{1}, \zeta_{2}, \lambda, \chi,\beta_k)$.

One can verify that the mapping $ T(w)$ maps the space (\ref{space}) into itself.

Analogously to the proof of Lemma \ref{lemma2}, we apply the fiber contraction theorem and show that
 there exists a unique fixed point
\begin{equation}
w = \eta_{1} v_{k1}(\zeta_{1},\chi,\lambda)+
\eta_{2} v_{k2}(\varphi_{2},\zeta_{1},\zeta_{2},\chi,\lambda)+
\eta_{3} v_{k3}(\varphi_{2},\zeta_{1},\zeta_{2},\chi,\lambda)\label{eq:w}
\end{equation}
of $T(w)$ in the neighborhood
of $(0,0)\in C^{l-4}(\mathbb{T}_{1}'\times\mathbb{T}_{3})\times I_{\lambda_{0}}$.

Functions in right-hand side of (\ref{eq:w}) are $C^{l-4}$ smooth and $2\pi$-periodic in
$\varphi_2, \zeta_1, \zeta_2,$ such that $\|v_{kj}\|_{C^{l-4}} \le M_2,$ where positive
constant $M_2$ does not depend on $\lambda,$ $\chi$ and $\beta_k.$

Respectively, there exist $\mu_{0}>0$ and $c_{0}>0$ such that for
all $0<\mu\le\mu_{0}$ and $\varepsilon\le c_{0}\sqrt{\mu}$ the system
(\ref{mm1j})--(\ref{mm3j}) possesses the invariant manifold\begin{equation}
b_{1}=\mu^{2} v_{k1}(\zeta_{1},\mu)+\varepsilon v_{k2}(\varphi_{2},\zeta_{1},\zeta_{2},\mu,\varepsilon)+
\frac{\varepsilon^{2}}{\mu} v_{k3}(\varphi_{2},\zeta_{1},\zeta_{2},\mu,\varepsilon).\label{eq:b11}
\end{equation}
Here we have used the same notations $v_{k1}$, $v_{k2},$ and $v_{k3}$
for the functions depending on parameters $\lambda,\mu$ in (\ref{eq:w})
and the corresponding functions depending on $\mu,\varepsilon$ in
(\ref{eq:b11}).

Therefore the system (\ref{mm1})--(\ref{mm2}) has
$2N$ integral manifolds
\begin{eqnarray}
& & \Pi_{j} = \{(\vartheta_{j}^{0}+\mu^{2}v_{j0}(\beta t,\mu)+\varepsilon v_{j1}(\varphi_{2},\beta t,\alpha t,\mu,\varepsilon)+ \nonumber \\
& & + \frac{\varepsilon^{2}}{\mu}v_{j2}(\varphi_{2},\beta t,\alpha t,\mu,\varepsilon):\ \varphi_{2}\in\mathbb{T}_{1},t\in\mathbb{R}\}.
\nonumber
\end{eqnarray}
The manifolds $\Pi_{2k}$, $k=1,...,N,$ are asymptotically stable
\cite{Samoilenko1991,Yi1993a}, i.e. there exists $\nu_{0}=\nu_{0}(\mu_{0},c_{0})$
such that if $\rho((\psi_{20},\varphi_{20}),\Pi_{2k}(t_{0}))\le\nu_{0}$
at time $t_{0}$ then there is a unique $\tilde{\varphi}_{20}$ such that
\begin{eqnarray}
 &  & |\psi_{2}(t,t_{0},\psi_{20},\varphi_{20})-\psi_{2}(t,t_{0},\vartheta_{2k}^{0}+v_{2k}(\tilde{\varphi}_{20},\beta t_{0},\alpha t_{0},\mu,\varepsilon),\tilde{\varphi}_{20})|\nonumber \\[2mm]
 &  & +|\varphi_{2}(t,t_{0},\psi_{20},\varphi_{20})-\varphi_{2}(t,t_{0},\vartheta_{2k}^{0}+v_{2k}(\tilde{\varphi}_{20},\beta t_{0},\alpha t_{0},\mu,\varepsilon),\tilde{\varphi}_{20})|\nonumber \\[2mm]
 &  & \le\mathcal{L}_{3}e^{-\mu^{2}\kappa_{3}(t-t_{0})}\left(|\varphi_{20}-
 \tilde{\varphi}_{20}|+|\psi_{20}-\vartheta_{2k}^{0}-v_{2k}(\tilde{\varphi}_{20},\beta t_{0},\alpha t_{0},\mu,\varepsilon)|\right), \label{sta2}\end{eqnarray}
 where $t\ge t_{0},$ constants $\mathcal{L}_{3}\ge 1$ and $\kappa_{3}>0$ are
independent on $\mu,\varepsilon,\alpha,$ $\rho(.,.)$ is metric in
$\mathbb{R}\times\mathbb{T}_{1},$ $\Pi_{2k}(t_{0})$ is the cross-section
of $\Pi_{2k}$ for $t=t_{0}.$

Since the function $b_{1}=v_{2k}(\varphi_{2},\zeta_{1},\zeta_{2},\mu,\varepsilon)$
is a smooth invariant manifold of (\ref{mm1j})--(\ref{mm3j}) we obtain
\begin{eqnarray}
 & & \frac{\partial v_{2k}}{\partial\varphi_{2}}\Biggl(\frac{\alpha_{0}}{\mu^{2}}+
 \bar{S}_{21}(\vartheta_{2k}^{0}+v_{2k},\mu) +
 \mu^2 \tilde{S}_{21}(\vartheta_{2k}^{0}+v_{2k},\zeta_1,\mu) + \nonumber \\
 & & + \varepsilon\tilde{S}_{22}(\vartheta_{2k}^{0}+v_{2k},\varphi_{2},\zeta_{1},\zeta_{2},\mu,\varepsilon)
 + \frac{\varepsilon^{2}}{\mu}\tilde{S}_{23}(\vartheta_{2k}^{0}+v_{2k},\varphi_{2},\zeta_{1},
 \zeta_{2},\mu,\varepsilon)\Biggl) \nonumber \\
  & & + \frac{\partial v_{2k}}{\partial\zeta_{1}}\frac{\beta}{\mu^{2}}+\frac{\partial v_{2k}}{\partial\zeta_{2}}\frac{\alpha}{\mu^{2}} = - \beta_{k}v_{2k}+\bar{G}_{2}(v_{2k})v_{2k}^{2}+\mu^{2}\tilde{S}_{11}(\vartheta_{2k}^{0}+v_{2k},\zeta_{1},\mu)\nonumber \\[2mm]
 &  & +\varepsilon\tilde{S}_{12}(\vartheta_{2k}^{0}+v_{2k},\varphi_{2},\zeta_{1},\zeta_{2},\mu,\varepsilon)+
 \frac{\varepsilon^{2}}{\mu}\tilde{S}_{13}(\vartheta_{2k}^{0}+
 v_{2k},\varphi_{2},\zeta_{1},\zeta_{2},\mu,\varepsilon).\label{intman}\end{eqnarray}
Taking into account (\ref{intman}) and making the change of variables
\[
b_{1}=v_{2k}(\varphi_{2},\zeta_{1},\zeta_{2},\mu,\varepsilon)+b_{2}\]
in (\ref{4mm1})--(\ref{4mm3}), we obtain the following system
(analogously as in the proof of Lemma~\ref{lemma2})
\begin{eqnarray*}
 &  & \frac{db_{2}}{d\tau}=\left[-\beta_{k}+T_{0}b_{2}+\mu^{2}T_{1}+\varepsilon T_{2}+\frac{\varepsilon^{2}}{\mu}T_{3}\right]b_{2}  \\
 &  & \frac{d\varphi_{2}}{d\tau}=\frac{\alpha_{0}}{\mu^{2}}+\bar{S}_{21}(\vartheta_{2k}^{0}+v_{2k}+b_{2},\mu)
 + \mu^2 \tilde {S}_{21}(\vartheta_{2k}^{0}+v_{2k}+b_{2},\zeta_1,\mu)  \\
 & &  + \varepsilon\tilde{S}_{22}(\vartheta_{2k}^{0}+v_{2k}+b_{2},\varphi_{2},\zeta_{1},\zeta_{2},\mu,\varepsilon)
 + \frac{\varepsilon^{2}}{\mu}\tilde{S}_{23}(\vartheta_{2k}^{0}+v_{2k}+
 b_{2},\varphi_{2},\zeta_{1},\zeta_{2},\mu,\varepsilon),\label{5mm2}\\
 &  & \frac{d\zeta_{1}}{d\tau}=\frac{\beta}{\mu^{2}},\quad\frac{d\zeta_{2}}{d\tau}=
 \frac{\alpha}{\mu^{2}},\label{5mm3}\end{eqnarray*}
with $C^{l-4}$-smooth functions $T_j$ of $(b_{2},\varphi_{2},\zeta_{1},\zeta_{2},\mu,\varepsilon)$,
periodic in $\varphi_{2},\zeta_{1},\zeta_{2}$ and uniformly bounded
for $b_{2}$ from some neighborhood of zero.

For sufficiently small $b_{2}$, $\mu^{2}$, $\varepsilon$, and $\varepsilon^{2}/\mu$,
we can obtain the uniform estimate $$\left|T_{0}b_{2}+\mu^{2}T_{1}+\varepsilon T_{2}+\frac{\varepsilon^{2}}{\mu}T_{3}\right|\le\beta_{k}/2.$$
Therefore the following inequality holds \begin{equation}
\left|b_{2}(t)\right|\le\left|b_{2}(t_{0})\right|e^{-\mu^{2}\frac{\beta_{k}}{2}(t-t_{0})}\label{eq:lin3}\end{equation}
 for all $b_{2}(t_{0})$ such that $\left|b_{2}(t_{0})\right|\le b_{20}$
with some $b_{20}>0$. Since $v_{2k}$ is a sum of three terms proportional
to $\mu^{2}$, $\varepsilon$, and $\varepsilon^{2}/\mu$ respectively
and $b_{20}$ is independent on these parameters, for small enough
$\mu^{2}$, $\varepsilon$, and $\varepsilon^{2}/\mu$, it holds $b_{20}-\|v_{2k}\|_{C}\ge\delta_{0}>0$.
Using $b_{1}=b_{2}+v_{2k}$, one can conclude that for all $b_{1}$
with $\left|b_{1}\right|\le\delta_{0}$ the inequality $\left|b_{2}\right|\le b_{20}$
and estimate (\ref{eq:lin3}) holds.

As a result, if $0<\mu\le\mu_{0}$ and $\varepsilon\le c_{0}\sqrt{\mu}$
and solution ($\psi_{2}(t),\varphi_{2}(t)$) of the system (\ref{mm1})
-- (\ref{mm2}) satisfies the condition $|\psi_{2}(t_{0})-\vartheta_{2k}^{0}|\le\delta_{0}$
at initial moment of time $t_{0}$ then
\begin{eqnarray}
& & \left|\psi_{2}(t)-\vartheta_{2k}^{0}-v_{2k}(\varphi_{2}(t),\beta t,\alpha t,\mu,\varepsilon)\right|
\nonumber \\[2mm]
& & \hspace{15mm }\le \left|\psi_{2}(t_{0})-\vartheta_{2k}^{0}-v_{2k}(\varphi_{2}(t_{0}),\beta t_{0},
\alpha t_{0},\mu,\varepsilon)\right|e^{-\mu^{2}\frac{\beta_{k}}{2}(t-t_{0})}\label{eq:ddd}
\end{eqnarray}
for all $t\ge t_{0}$.

Inequalities (\ref{sta2}) and (\ref{eq:ddd}) assure the exponential
attraction of all solutions of (\ref{mm1})--(\ref{mm2}) that start
at $t=t_{0}$ from a small neighborhood of the unperturbed manifold
$\psi_{2}=\vartheta_{2k}^{0}$ to solutions of the perturbed manifold
$\Pi_{2k}$ according to the estimation (\ref{sta2}).

Considering the system (\ref{mm1})--(\ref{mm2}) in the neighborhood
of the manifolds $\Pi_{2k-1}$, $k=1,\dots,N$, we obtain similarly
that these manifolds are exponentially unstable according to (\ref{sta2-}).
\begin{corollary}
\label{cor6} The system (\ref{m1})--(\ref{m2}) has $2N$ integral
manifolds \begin{eqnarray}
\mathcal{P}_{j}^{0}=\{(\beta t+\vartheta_{j}^{0}+\tilde{v}_{j}(\varphi,\beta t,\alpha t,\mu,\varepsilon),\varphi,t):\ \varphi\in\mathbb{T}_{1},t\in\mathbb{R}\},\label{ma3}\end{eqnarray}
where the $C^{l-4}$-smooth function\[
\tilde{v}_{j}=\mu^{2}\tilde{v}_{j0}(\beta t,\mu)+\varepsilon\tilde{v}_{j1}(\varphi,\beta t,\alpha t,\mu,\varepsilon)+\frac{\varepsilon^{2}}{\mu}\tilde{v}_{j2}(\varphi,\beta t,\alpha t,\mu,\varepsilon),\]
is periodic in $\varphi,\beta t,$ and \textup{$\alpha t$.} On the
manifolds (\ref{ma3}), the system (\ref{m1})--(\ref{m2}) reduces
to \begin{align}
\frac{d\varphi}{dt} & =\alpha_{0}+\mu^{2}S_{21}(\beta t+\vartheta_{j}^{0}+\tilde{v}_{j},\beta t,\mu)+\varepsilon\mu^{2}S_{22}(\beta t+\vartheta_{j}^{0}+\tilde{v}_{j},\varphi,\beta t,\alpha t,\mu,\varepsilon)\nonumber \\
 & \hspace{15mm}+\varepsilon^{2}\mu S_{23}(\beta t+\vartheta_{j}^{0}+\tilde{v}_{j},\varphi,\beta t,\alpha t,\mu,\varepsilon).\label{ma3na}\end{align}
The manifolds corresponding to $j=2k,k=1,...,N,$ are exponentially stable
for $t\to+\infty$ and the manifolds corresponding to $j=2k-1,k=1,...,N,$
are exponentially stable for $t\to-\infty.$ \end{corollary}
\begin{lemma}
\label{lemma7}There exist $\mu_{0}>0$ and $c_{0}>0$ such that for
all $0<\mu\le\mu_{0}$ and $\varepsilon\le c_{0}\sqrt{\mu}$ the solutions
of (\ref{mm1})--(\ref{mm2}) have the following properties:

(i) if a solution $(\psi_{2}(t),\varphi_{2}(t))$ at a certain time
$t=t_{0}$ has the value $\psi_{2}(t_{0})=\vartheta_{2k-1}^{0}+\delta,$
then it reaches the value $\psi_{2}(t_{0}+T)=\vartheta_{2k}^{0}-\delta$
after a finite time interval of the length $T=T(\delta,\mu,\varepsilon)$;

(ii) if a solution $(\psi_{2}(t),\varphi_{2}(t))$ at a certain time
$t=t_{0}$ has the value $\psi_{2}(t_{0})=\vartheta_{2k+1}^{0}-\delta,$
then it reaches the value $\psi_{2}(t_{0}+T)=\vartheta_{2k}^{0}+\delta$
after a finite time interval of the length $T=T(\delta,\mu,\varepsilon)$.
(Here we identify $\vartheta_{2N+1}^{0}$ with $\vartheta_{1}^{0}$
and $\vartheta_{0}^{0}$ with $\vartheta_{2N}^{0}$).
\end{lemma}
\textbf{Proof.} Let us consider the interval $(\vartheta_{2k-1}^{0},\vartheta_{2k}^{0}).$
The intervals $(\vartheta_{2k}^{0},\vartheta_{2k+1}^{0})$ can be
considered similarly. Denote \[
m=\min_{\xi\in[\vartheta_{2k-1}^{0}+\delta,\vartheta_{2k}^{0}-\delta]}({G}(\xi)-\Delta)>0.\]
The right-hand side of  (\ref{mm1}) can be estimated as follows
\begin{eqnarray*}
\frac{d\psi_{2}(t)}{dt}=\mu^{2}({G}(\psi_{2})-\Delta)+\mu^{4}\tilde{S}_{11}+\varepsilon\mu^{2}\tilde{S}_{12}+\varepsilon^{2}\mu\tilde{S}_{13}\ge\mu^{2}\left(m-m_{0}\right),
\label{3l1}
\end{eqnarray*}
where $m_{0}=m_{0}(\mu_{0},c_{0})=\sup\left(\mu^{2}\tilde{S}_{11}+\varepsilon\tilde{S}_{12}+\frac{\varepsilon^{2}}{\mu}\tilde{S}_{13}\right).$
By choosing sufficiently small $\mu_{0}$ and $c_{0}$ , one can obtain
$m_{0}<m$. Hence \begin{eqnarray*}
 &  & \psi_{2}(t)\ge\psi_{2}(t_{0})+\mu^{2}(m-m_{0})(t-t_{0}),\\
 &  & t-t_{0}\le\frac{\psi_{2}(t)-\psi_{2}(t_{0})}{\mu^{2}(m-m_{0})}\le
 \frac{\vartheta_{2k}^{0}-\vartheta_{2k-1}^{0}}{\mu^{2}(m-m_{0})}
 \le \frac{2\pi}{\mu^{2}(m-m_{0})}
  =T(\delta,\mu,\varepsilon).\end{eqnarray*}

\textbf{Proof of Theorem \ref{theorem01}.} Theorem \ref{theorem01}
follows from Lemma~\ref{lemma2} and the following chain of coordinate
changes: averaging transformations from section \ref{sec:Averaging},
(\ref{eq:polar}), and the local coordinates (\ref{eq:zamina}) in the neighborhood of the
invariant manifold $\mathcal{T}_2$.

\textbf{Proof of Theorem \ref{theorem02}}. In Lemma \ref{lemma4},
the existence and local stability properties of the integral manifolds
$\Pi_{j}$, $j=1,...,2N$ have been proved. The integral manifolds
$\mathfrak{N}_{j}$ correspond to the manifolds $\Pi_{j}$ after the
averaging and transformations (\ref{eq:polar}) and (\ref{eq:zamina}).

It has been proved in Theorem \ref{theorem01} that all solutions
from some neighborhood of the torus $\mathcal{T}_{2}$ are approaching
the perturbed integral manifold $\mathfrak{M}(\alpha,\beta,\gamma)$.
Therefore, for the proof of the statement 2 of Theorem \ref{theorem02}
it is enough to show that the solutions on this manifold are approaching
the solutions on one of the manifolds $\mathfrak{N}_{j}$.

Let us fix any positive $\varepsilon_1.$ 
For the set $S$ of singular values of $G$ we define two following sets:
$$\mathcal{B}(\varepsilon_1) = \{ g \in [G_-, G_+]; \ dist(g, S) \ge \varepsilon_1\},$$
$$\mathcal{A}(\varepsilon_1) = \{ \theta \in [0,2\pi]: \  G(\theta) \in \mathcal{B}(\varepsilon_1)\}.$$

Taking into account that the sets  $\mathcal{B}(\varepsilon_1)$ and $\mathcal{A}(\varepsilon_1)$ are compact
one can prove that there exists a positive constant $\varsigma$ such that
\begin{equation} \label{estpro}
\left|\frac{d G(\theta)}{d\theta}\right| \ge \varsigma
\quad {\rm for \ all} \quad  \theta \in \mathcal{A}(\varepsilon_1).
\end{equation}

Let us consider the system (\ref{mm1})--(\ref{mm2}), which describes
the dynamics on the manifold $\mathfrak{M}(\alpha,\beta,\gamma)$.
For any $\alpha, \gamma$ and $\beta$ satisfying (\ref{eq:cond2}) and (\ref{eq:cond3}) there exists
a finite number of points $\vartheta_j^0, j = 1,...,2N,$ (solutions of the equation
$\beta - \beta_0 = \mu^2  G(\theta)$), which define the
integral manifolds $\Pi^0_j, j = 1,...,2N,$ of the averaged system (\ref{mm01}) - (\ref{mm02}).
Note that number $N$ depends on the parameters $\alpha, \gamma$ and $\beta.$

By Lemma \ref{lemma4}, for fixed $\Delta \in \mathcal{B}(\varepsilon_1)$, there exist $\mu_{0}>0$ and $c_{0}>0$ such that
for all $0<\mu\le\mu_{0}$ and $\varepsilon\le c_{0}\sqrt{\mu}$ the system
(\ref{mm1})--(\ref{mm2}) has $2N$ integral manifolds
$\Pi_{j}.$ Due to the uniform estimate (\ref{estpro}), it follows from the proof of Lemma
\ref{lemma4} that constants $\mu_{0}>0$ and $c_{0}>0$ can be chosen the same for all
$\Delta \in \mathcal{B}(\varepsilon_1),$ and therefore for all
$\alpha, \gamma$ and $\beta$ satisfying (\ref{eq:cond2}) and (\ref{eq:cond3}).

All the manifolds $\Pi_{2k}$ are asymptotically stable in the sense of
the formula (\ref{sta}) and the manifolds $\Pi_{2k-1}$, $1\le k\le N$
are asymptotically unstable in the sense of the formula (\ref{sta2-}).
Therefore, if $|\psi_{20}-\vartheta_{2k-1}^{0}|<\delta_{0}$ and $(\psi_{20},\varphi_{20})\notin\Pi_{2k-1}$
then \begin{eqnarray}
|b_{2}(t)|\ge K_{2}e^{\mu^{2}\kappa_{2}(t-t_{0})}|b_{2}(t_{0})|,\ t\ge t_{0},\label{exp3}\end{eqnarray}
where $K_{2}\ge1,$ $\kappa_{2}>0$ are some constants and  
$b_{2}(t)=\psi_{2}(t,t_{0},\psi_{20},\varphi_{20},\varepsilon,\mu)-\psi_{2}(t,t_{0},
\vartheta_{2k-1}^{0}+\varepsilon v_{2k-1}(\varphi_{20},\beta t_{0},\alpha t_{0},
\varepsilon,\mu),\varphi_{20},\varepsilon,\mu).$

It follows from (\ref{exp3}) that on a finite time interval $T$
depending on values $\psi_{20}$ and $\mu,\varepsilon$ the solution
$(\psi_{2}(t),\varphi_{2}(t))$ of (\ref{mm1})--(\ref{mm2}), whose initial value $(\psi_{2}(t_{0}),\varphi_{2}(t_{0}))$
for $t=t_{0}$ does not belong to the manifold $\Pi_{2k-1}$, i.e.
\[
\psi_{2}(t_{0})\neq\vartheta_{2k-1}^{0}+v_{2k-1}(\varphi_{2}(t_{0}),\beta t_{0},\alpha t_{0},\varepsilon,\mu),\]
and $|\psi_{2}(t_{0})-\vartheta_{2k-1}^{0}|<\delta_{0}$, reaches
the boundary of $\delta_{0}$-neighborhood of $\vartheta_{2k-1}^{0},$
more exactly, values $\psi_{2}(t_{1})=\vartheta_{2k-1}^{0}-\delta_{0}$
or $\psi_{2}(t_{2})=\vartheta_{2k-1}^{0}+\delta_{0}.$

Then, by Lemma \ref{lemma7}, on a finite time interval, this solution
reaches $\delta_{0}$-neighborhood of point $\vartheta_{2k}^{0}$
or, respectively, $\delta_{0}$-neighborhood of point $\vartheta_{2k+2}^{0}$,
where $\delta_{0}$ is defined from Lemma \ref{lemma4}.

Next, by Lemma \ref{lemma4}, as $t$ further increases, the solution
is attracted to one of the stable integral manifolds $\Pi_{2k}$ or
$\Pi_{2k+2}.$

As a result, solutions $(\psi(t),\varphi(t))$ of the system (\ref{m1})
-- (\ref{m2}) that, at initial point $t=t_{0}$ do not belong to the
unstable integral manifolds $\mathcal{P}_{2k-1},k=1,...,N,$ i.e., \[
\psi(t_{0})\neq\vartheta_{2k-1}^{0}+\beta t_{0}+\tilde{v}_{2k-1}(\varphi(t_{0}),\beta t_{0},\alpha t_{0},\mu,\varepsilon),\]
are attracted for $t\ge t_{0}$ to solutions $(\bar{\psi}(t),\bar{\varphi}(t))$
on one of the stable integral manifolds $\mathcal{P}_{2k}$\begin{align*}
\bar{\psi}(t) & =\beta t+\vartheta_{2k}^{0}+\tilde{v}_{2k}(\varphi(t),\beta t,\alpha t,\mu,\varepsilon),\\
\bar{\varphi}(t) & \mbox{\,\, is a solution of system (\ref{ma3na}) for }j=2k,\end{align*}
so that \[
|\psi(t)-\bar{\psi}(t)|+|\varphi(t)-\bar{\varphi}(t)|\le\mathcal{L}_{2}e^{-\mu^{2}\kappa_{2}(t-T)}
\left(|\psi(T)-\bar{\psi}(T)|+|\varphi(T)-\bar{\varphi}(T)|\right),\ t\ge T,\]
 for some $T=T(\psi(t_{0}),\mu,\varepsilon)$ and some $\mathcal{L}_{2}\ge1$.

If a solution $(\psi(t),\varphi(t))$ of (\ref{m1})--(\ref{m2})
at the initial point $t=t_{0}$ belongs to one of integral manifolds $\mathcal{P}_{2k+1}$
then this solution has the following form
\begin{eqnarray*}
 &  & \psi(t)=\beta t+\vartheta_{2k+1}^{0}+\tilde{v}_{2k+1}(\varphi(t),\beta t,\alpha t,\mu,\varepsilon),\nonumber \\[1mm]
 &  & \varphi(t)\mbox{\,\, is a solution of system (\ref{ma3na}) for }j=2k+1.\label{exp5}
\end{eqnarray*}

Using the last formulas and Lemma~\ref{lemma2}, we conclude that
any solution $(h(t),\psi(t),\varphi(t))$ of (\ref{ss1})
-- (\ref{ss3}) that starts from the $\nu_{0}$-neighborhood
of the integral manifold $\mathcal{T}_2$ is attracted
to one of the solutions $(\bar{h}(t),\bar{\psi}(t),\bar{\varphi}(t))$
on the integral manifold $\mathfrak{M}_{\mu,\varepsilon}$ such that
\begin{eqnarray*}
 &  & \bar{h}(t)=u(\bar{\psi}(t),\bar{\varphi}(t),\beta t,\alpha t,\mu,\varepsilon),\nonumber \\[1mm]
 &  & \bar{\psi}(t)=\beta t+\vartheta_{j}^{0}+\tilde{v}_{j}(\bar{\varphi}(t),\beta t,\alpha t,\mu,\varepsilon),\nonumber \\[1mm]
 &  & \bar{\varphi}(t)\mbox{\,\, is a solution of system (\ref{ma3na})}\label{62}
\end{eqnarray*}
with some $j,\ 1\le j\le2N.$ More exactly, there exist constants $L \ge1, \kappa > 0$ and $T=T(h(t_{0}),\psi(t_{0}),\varphi(t_{0}))\ge t_{0}$
such that for $t\ge T:$
\begin{eqnarray*}
 &  & |h(t)-u(\bar{\psi}(t),\bar{\varphi}(t),\beta t,\alpha t,\mu,\varepsilon)|+|\psi(t)-\bar{\psi}(t)|+|\varphi(t)-\bar{\varphi}(t)|\le\nonumber \\
 &  & \le Le^{-\kappa(t-T)}\Bigl(|h(T)-u(\bar{\psi}(T),\bar{\varphi}(T),\beta T,\alpha T,\mu,\varepsilon)|+\nonumber \\
 &  & +|\psi(T)-\bar{\psi}(T)|+|\varphi(T)-\bar{\varphi}(T)|\Bigl).\label{exp6}
\end{eqnarray*}

\textbf{Proof of Theorem \ref{thm:main}}. Under the conditions of
Theorem \ref{thm:main}, the conditions of Theorem \ref{theorem02}
are satisfied. Therefore, every solution $(x(t),y(t))$ of the system
(\ref{01})--(\ref{02}) that at a certain moment of time $t_{0}$
belongs to a $\delta$-neighborhood of the torus $\mathcal{T}_{2}$
tends to some solution on one of the integral manifolds $\mathfrak{N}_{j}(\alpha,\beta,\gamma), j =1,...,2N.$ Hence,
for any $\varepsilon>0$ the following inequality holds \[
\biggl\| x(t)-x_{0}(\beta t+\vartheta_{j}^{0})\biggr\|+\biggl||y(t)|-|y_{0}(\beta t+\vartheta_{j}^{0})|\biggr|<\epsilon\]
with some $1\le j\le N$ for all moments of time starting from $T(x(t_{0}),y(t_{0}))$.

\section*{Acknowledgments}
 The authors gratefully acknowledge the  scientific cooperation and
the helpful discussions with K. R. Schneider over many years which lead to
the two preliminary versions \cite{Schneider2005,Samoilenko2005} of the
present research.

\medskip
Received xxxx 20xx; revised xxxx 20xx.
\medskip

\end{document}